\documentclass[12pt,a4paper, fleqn] {amsart}
\usepackage{amsmath,amscd,amsfonts,eucal,latexsym,amssymb,mathrsfs}
\usepackage{graphicx}
\newtheorem{theo}{Theorem}[section]
\newtheorem{cor}[theo]{Corollary}
\newtheorem{lem}[theo]{Lemma}
\newtheorem{rem}[theo]{Remark}

\def\Cal{\mathcal}

\def\cC{{\Cal C}}

\def \bre {{\mathbb {R}}}
\def\bar{\overline}
\def\bna{\Bbb N}

\def\llim{\mathop{\longrightarrow}}

\def\vo{V^{(0)}}
\def\von{V^{(1)}}
\def\vn#1{V^{(#1)}}

\def\s,{\quad $\,$}

\font\pio=cmr10 scaled 1230
\def\empti{{\text{\rm \pio \o}}}

\def\qua{\quad$\,$}

\def\s,{\quad $\,$}

\def\witi{\widetilde}

\def\disp{\displaystyle}
\def\cite#1{[#1]}

\def\diam{\text{\rm diam}}
\def\asd{$\alpha$-scaling distance}
\def\pal{^{(\alpha)}}

\def\Dal{D_{\alpha}}
\def\wdal{\witi d_{\alpha}}

\def\path{ \big({\Cal Pa}\big)} 
\def\tpath{ \big(\witi{\Cal Pa}\big)} 
\def\jet{ \widehat J\times W^*} 
\def \alw{\alpha_ {\bf \bar w}}
\def\cp{\Cal P}
\def\tw{ w^*}

\def\balpha{{\bf \alpha}}

\begin{document}

\title{
 SCALING Distances on Finitely Ramified Fractals}

\maketitle
 \author{Roberto Peirone \footnote{
  : e-mail: {\sf peirone@axp.mat.uniroma2.it}, Phone: +39\,0672594610,
    Fax: +39\,0672594699}},
  {Universit\`a di Roma-Tor Vergata,
   Dipartimento di Matematica,        
Via della Ricerca Scientifica, 00133 Roma, Italy}

\begin{abstract}
In previous papers by A. Kameyama and by J. Kigami 
distances on fractals have been discussed having two different 
but similar properties.
One property is that the maps defining the fractal are Lipschitz
of prescribed constants less than $1$, the other is that the diameters of the 
copies of the fractal  are asymptotic to prescribed scaling factors.
In this paper, on  a large class of finitely ramified fractals,
we  prove that  these two problems are equivalent and give
 a necessary and sufficient condition for the existence of such distances.
Such a condition
is expressed in terms of asymptotic behavior of the product of certain matrices
associated to the fractal.  
\end{abstract}


\section{Introduction}
Fractals are very irregular  mathematical objects.
Over the last three decades,
 they have been investigated rather intensively.
 The notion of fractal  is very general, and in this paper 
 a specific but rather general class of fractals is considered.
 More ore less, we consider the so-called self-similar fractals,
 i.e., geometric objects having the property of
 containing copies of them at arbitrarily small
 scales. A concrete and relatively general
 class of self-similar fractals
 is the following:  we are given finitely many contractive similarities 
 $\psi_1,...,\psi_k$ in $\bre^n$, that is

 $$d\big(\psi_i (x),\psi_i(x')\big)=\witi\alpha_i d(x,x') \eqno (1.1)$$
 
 for every $x,x'\in K$, where $\witi\alpha_i$ are constant
 lying in $]0,1[$. The self-similar fractal
  $K$ (generated by the similarities $\psi_i$) is the
   set  invariant with respect to such a set
 of similarities, more precisely, $K$ is the only nonempty compact
 subset $K$  of $\bre^n$ such that
 
 $$K=\bigcup\limits_{i=1}^k \psi_i(K).\eqno (1.2)$$
 
 More generally, we could define a self-similar fractal 
 as a compact metric space, or also as a compact topological space
 satisfying (1.2) where $\psi_i$ are maps satisfying
 certain conditions (see \cite{3}, \cite{4}).
 Examples of self-similar fractals are the Cantor set,
 the Koch curve, the (Sierpinski) Gasket, the (Sierpinski) Carpet,
 the Vicsek Set, the (Lindstr\o m) Snowflake. See for example 
 \cite{5},  \cite{7} for the description of such fractals.  

 \qua  Here, following
\cite{3} and \cite{4}, we will say that $\balpha:=(\alpha_1,...,\alpha_k)$ 
is a {\it polyratio} if $\alpha_i\in ]0,1[$ for every $i=1,...,k$. 
When  the self-similar fractal is defined by (1.2) in a metric
 space where $\psi_i$ satisfy (1.1), then, by definition
 the maps $\psi_i$ are $\witi \alpha_i$-Lipschitz.
 More generally, 
 one could ask for what  polyratios $\balpha:=(\alpha_1,...,\alpha_k)$  
  there exists a distance 
 $d$ on $K$ such that the maps $\psi_i$ are $\alpha_i$-Lipschitz
 with respect to $d$.
 This problem is  discussed in \cite{3} and in \cite{4}.
  When this happens, following
\cite{3}, \cite{4}, we will say that $\alpha$ is a {\it metric polyratio} (on $K$), and
also, we will say that $d$ is an $\balpha$-{\it self-similar distance}.  
 
 \qua In \cite{4} a pseudodistance $D_{\balpha}$ is constructed on a 
 general class of self-similar fractals for every polyratio, and
 it is proved that such  if $\balpha$ is  a metric polyratio,
 then such  a pseudodistance  is a distance, that is, 
 if an $\balpha$-self-similar distance exists, then $D_{\balpha}$
 is an $\alpha$-self-similar distance. Moreover, $D_{\balpha}$ induces 
 on $K$ the same topology as the original one.

 \qua A similar  problem
  has been  studied in  \cite{6}, in  connection
 to problems related to analysis on fractals. 
 Note that
if $\balpha$ is a metric polyratio, then  
 the diameter of the copy $\psi_{i_1}\circ\cdots\circ \psi_{i_m}(K)$
 with respect to 
 an $\balpha$-self-similar distance $d$ is not greater than  $
\alpha_{i_1}\cdots\alpha_{i_m} \diam(K)$, diam$(K)$ denoting the diameter
with respect to $d$. 
We will say that $\balpha$ is an {\it asymptotic metric polyratio} (on $K$) or shortly
{\it as.$\,$metric polyratio} if 
  there exista a distance $d$ on $K$ and positive constants
  $c_{1,\alpha}, c_{2,\alpha}$ such 
$$c_{1,\alpha} \alpha_{w} \diam_d(K)\le 
 \diam_d\big(\psi_{i_1}\circ\cdots\circ \psi_{i_m}(K)
 \big)\le c_{2,\alpha} \alpha_{w} \diam_d(K) \eqno (1.3)$$ 
 
 where $\alpha_{w}=\alpha_{i_1}\cdots \alpha_{i_m}$, for every $i_1,...,i_m=1,...,k$.
 We say in such a case that the  distance $d$ is an $\alpha$-{\it scaling distance}.
 
 \qua The analysis on fractals
 deals  with notions like Laplace operator, Dirichlet integral,
 Heat equation.  
  The problem of defining
 such notions
  is nontrivial since we cannot define the derivative in the usual
 sense on fractals, as they have usually an empty interior.
 Analysis on fractals
 has been developed   mainly on the {\it finitely ramified fractals},
 and more specially on the P.C.F.  self-similar sets. 
 The  Gasket, 
 the Vicsek Set and  the   Snowflake are examples of P.C.F.  self-similar sets,
 while the Carpet is not a P.C.F.  self-similar sets and neither is finitely ramified.
  However, also on certain  infinitely ramified fractals, e.g., the Carpet,
  analysis has been developed.
 Standard text-books on analysis on fractals are \cite{5} and  \cite{7},
 where also the precise definition of P.C.F.  self-similar sets is given. 

\qua  In \cite{6}, the problem of what polyratios are 
 as. metric polyratios is discussed as related to
 problems concerning Dirichlet forms and Heat Kernel.
 There,  a rather general class of self-similar fractals
 (not only finitely ramified)
 is discussed, and some conditions are given.
 In particular, it is proved that in any case
  many polyratios are as. metric.
 However,  the conditions given in \cite{6}
 are not necessary and sufficient.

 \qua In this paper, I restrict the class of fractals (more or less I consider
 the connected P.C.F. self-similar sets). See Section 2 for the details.
 For such class of fractals, I 
 prove that the notions of metric polyratio 
 and as. metric polyratio are in fact the same. Moreover, I
 give a necessary and sufficient condition
 for a polyratio being metric (or as. metric).
 This condition is based on a finite set of
 special  matrices  which are described
  in Section 4. 
  These matrices are related to the paths
  on the fractal. The condition is that this set of matrices
  satisfies a special property, which is strictly related to the
  notion of {\it joint spectral radius}, or better of {\it joint spectral subradius}.
  Joint spectral radius and joint spectral subradius
  are notions that  generalize the notion of spectral radius
  to the case of a finite set of matrices.
  More precisely, given a finite set $E$ of $n\times n$ matrices,
  the spectral radius (resp. spectral subradius) is defined as
  
  $$\lim\limits_{h\to +\infty}\max
  \big\{ ||A_{i_1}\cdots A_{i_h}||: A_{i_1},...,A_{i_h}\in E\big\},$$
   $$\big(\text{\rm resp.}\ 
  \lim\limits_{h\to +\infty}\min
  \big\{ ||A_{i_1}\cdots A_{i_h}||: A_{i_1},...,A_{i_h}\in E\big\}\big).$$
  
  A text-book on joint  spectral radius and joint  spectral subradius
  is \cite{2}.

  \qua Section 5 is devoted to prove the condition given here (Theorem 5.7).
    The condition is strictly related to the statement
  that such spectral objects are greater than or equal to $1$.
  However, it is not equivalent.  Namely, we require that a set $E$ of
  matrices with nonnegative entries satisfies
  
  $$||A_{i_1}\cdots A_{i_h}(e_j)||\ge c$$
where $c$ is a positive constant independent of $h$, of
the matrices $A_{i_1},...$, $A_{i_h}\in E$ and of the vector
$e_j$ of the canonical basis.  
  Note that usually, at least to my knowledge,
  it is difficult to evaluate  joint spectral radius and  joint spectral subradius
   of a  finite set of matrices. So, I expect that in the general
   case an explicit and effective condition for 
   a polyratio being metric  on the given fractal
   could be hard to find.
   However, this can be done for some specific fractals.
   In Section 6,  I give explicit necessary and sufficient conditions
   on the Gasket and on the Vicsek set. I expect
   that similar explicit conditions can be given for fractals having a simple
   structure, and for   more complicated fractals, if the
   factors $\alpha_i$ have some good symmetry properties.

\section{Fractals}

In this Section, we describe the construction of a fractal, following
more or less the approach of \cite{5}. The results of this section are
standard (either known or simple consequences
of known results), but I prefer to recall them.
 Let $k$ be an integer greater than $1$. Define

 $$W_m=\{1,...,k\}^m,\quad \witi W_m:=\bigcup\limits_{l\le m} W_l,
 \quad W^*:=\bigcup\limits_{m\in\bna} W_m, $$
$$ W(=W_{\infty}):= \{1,...,k\}^{\bna}.$$
 
 If $\tw\in W^*$, let $|\tw|:=m$ if $\tw\in W_m$, and we say that $w$ 
 is a word and $m$ is the length
 of $w$. We equip $W$ with the product topology $\{1,...,k\}
 ^{\bna}$, where on $\{1,...,k\}$ 
 we put the discrete distance.  Note that the unique word in $W_0$ is 
 the empty word $\empti$.
 
\qua  Let $\sigma_i:W\to W$ be defined by
 $\sigma_i(w)=iw$, where if $w=(i_1, i_2, i_3....)$, we set $iw=(i, i_1, i_2, i_3...)$.
 If $w\in W$, or $w\in W_m$ with $m\ge \bar m$,
 $w=(i_1,i_2,i_3,....)$, let $w_{(	\bar m)}\in W_{\bar m}$
  be defined by $w_{(\bar m)}=
 (i_1,...,i_{\bar m})$.  
 If $\tw\in W^*$ and $m\le |\tw|$, we say that $\tw_{(m)}$ is  a
 {\it segment} of $\tw$. We say that two words 
 $\tw$ and ${\tw}'$ are {\it incomparable}
 if neither $\tw$ is a segment of ${\tw}'$ nor 
 ${\tw}'$ is a segment of $\tw$. 
 Denote by $i^{(m)}$ the element of $W_m$ 
 of the form $(i,...,i)$ where $i$ is repeated $m$
 times, for $m\in\bna\cup\{\infty\}$.
 
 \qua  Let $(K,\bar d)$ be a compact metric space.
 We say that $K$ is a self-similar fractal is there exists
  a continuous 
  map $\pi$ from $W$ onto $K$ and  
 continuous one-to-one maps $\psi_i$, $i=1,...,k$
 from $K$ into itself such that 
 
 $$\psi_i\circ\pi=\pi\circ\sigma_i\quad\forall\, i=1,...,k.\eqno (2.1)$$
 
 If $\tw=(i_1,...,i_m)$, let $\sigma_{\tw}=\sigma_{i_1}\circ\cdots \circ\sigma_{i_m}$,
 $\psi_{\tw}=\psi_{i_1}\circ\cdots \circ\psi_{i_m}$. It follows that
 
$$\psi_{\tw}\circ\pi=\pi\circ\sigma_{\tw}\quad\forall\, \tw\in W^*.\eqno (2.2)$$ 
 
 For every $w\in W$ we have
 
 $$\{\pi(w)\}=\bigcap\limits_{m=1}^{\infty}\psi_{w_{(m)}}(K). \eqno (2.3)$$
 
 In fact, $w=\sigma_{w_{(m)}}(w')$ for some $w'\in W$, then,
 for every $m=1,2,....$ we have:
 
 $$ \pi(w)=\pi\big(\sigma_{w_{(m)}}(w')\big)=\psi_{w_{(m)}}\big(\pi(w')\big)
\in  \psi_{w_{(m)}}(K)$$
 
 and the inclusion $\subseteq$ in (2.3 is proved. On the other hand,
 if $y\in  \bigcap\limits_{m=1}^{\infty}\psi_{w_{(m)}}(K)$, then
 for every $m=1,2,...$, there exists  $w'_m\in W$ such that
 $y=\psi_{w_{(m)}}\big(\pi(w'_m)\big)=\pi\big(\sigma_{w_{(m}}(w'_m)\big)$.
 Since, by the definition of the topology on $W$, we have
 $\sigma_{w_{(m)}}(w'_m)\llim\limits_{m\to +\infty} w$, in view of the continuity
 of $\pi$, we have $y=\pi(w)$, and (2.3) is completely proved.
 
 \qua Let now 

$$E_w=\psi_w(E)\quad\forall\, w\in W^*, E\subseteq K.$$
By (2.1) we have

$$\bigcup\limits_{i=1}^k K_i=\bigcup\limits_{i=1}^k \psi_i\big(\pi(W)\big)=$$
$$=\bigcup\limits_{i=1}^k\pi\big(\sigma_i(W)\big)=
\pi\Big(\bigcup\limits_{i=1}^k\sigma_i(W)\Big)=\pi(W)=K.$$

More generally, for every $m=1,2,3....$ we have

$$K=\bigcup\limits_{w\in W_m}K_w.\eqno (2.4)$$
$$K_{i_1,...,i_{m-1}}=\bigcup\limits_{i_m=1}^k K_{i_1,...,i_{m-1},i_m}
\eqno (2.4')$$

As  a consequence, we have

$$K_{i_1,...,i_m}\supseteq K_{i_1,...,i_{m-1}} 
\quad \forall\, i_1,...,i_m=1,...,k.\eqno (2.4'')$$

For the following we will require additional properties on the fractal.
We require that the fractal is a P.C.F. self-similar set
with a little additional property similar to that required in \cite{7}.
Suppose  $\psi_i$ has a unique fixed point
which we denote by  $P_i$, and let $\witi V=\{P_i, i=1,...,k\}$.
 Assume there exists a subset $V=\vo=\{P_1,...,P_N\}$ of $\witi V$,
 of $N$ elements with $2\le N\le k$.


\qua The sets of the form $V_{i_1,...,i_m}$ will bee called {\it $m$-cells},
and the sets of the form $K_{i_1,...,i_m}$ will be called {\it $m$-copies}.
We will require that the intersection of two different
$m$-cells amounts to the intersection of the corresponding
$m$--copies. The Sierpinski Carpet is a fractal that does not fill
such a requirement.
More precisely, suppose that,  when 
$\tw,{\tw}'\in W^*$, $\tw\ne {\tw}'$, $|\tw| = |{\tw}'|$, then

$$ K_{\tw}\cap K_{{\tw}'}
=V_{\tw}\cap V_{{\tw}'},\eqno(2.7)$$
$$\#\big(K_{\tw}\cap K_{{\tw}'}\big)\le 1.\eqno(2.7')$$

Requirement (2.7) is more or less the finite ramification property. 
Requirement (2.7$'$) is possibly not strictly necessary, but simplifies many
arguments and is satisfied by almost all the finitely ramified
fractals considered in the literature.
Moreover, we require

$$\text{If}\ i=1,...,k, \ j,h=1,....,N\ \text{and}\ \psi_i(P_j)=P_h\ \text{then}\ i=j=h.
\eqno (2.8)$$

 Note that (2.8) in particular implies that $P_j\ne P_{j'}$ if $j\ne j'$. Let 
 
$$ \widehat J:=\big\{(j_1,j_2):j_1,j_2=1,...,N, j_1\ne j_2\big\},$$
 
  and note
 that $\#(\widehat J)=\witi M:=N(N-1)$. Let
 
 $$\vn m=\bigcup\limits_{i_1,...i_m=1}^k V_{i_1,...i_m}, 
 \quad \vn{\infty}=\bigcup\limits_{m=0}^{\infty} \vn m.$$
 
 Note that $\vn m\subseteq \vn{m+1}$ for every positive integer $m$.
 I now prove some lemmas useful in the sequel.

\begin{lem}\label{2.1}
For every $\tw\in W^*$, the map $\psi_{\tw}$ sends $K\setminus \vo$ into itself.
More precisely, if $\tw\in W_m$, $m>0$, and $\psi_{\tw}(Q)=P_j\in\vo$, $Q\in K$, then
$Q=P_j$  and $\tw=j^{(m)}$.
\end{lem}
\begin{proof}
Let $Q\in K$ and suppose  and $\psi_{\tw}(Q)=P_j\in\vo$. Then
$\tw=j^{(m)}$ and consequently  $Q=P_j$. In fact in the opposite case,
as $P_j=\psi_{j^{(m)}}(P_j)$,
by (2.7) we have $Q\in\vo$, and by (2.8) and an inductive
argument we have $\tw=j^{(m)}$, a contradiction.
 \end{proof}

Note that (2.7) holds under the hypothesis $|\tw|=|{\tw}'|$, and that 
(2.7) is no longer valid if $\tw$, ${\tw}'$ are two arbitrary different words.
In fact, if ${\tw}'$ is a segment of $\tw$, then $K_{\tw}\cap K_{{\tw}'}= K_{\tw}$
by (2.4$''$). However,
as we see now, this is the only case in which (2.7) does not hold.

\begin{lem}\label{2.2}
If $\tw$ and ${\tw}'$ are two incomparable words, then

\centerline{$K_{\tw}\cap K_{{\tw}'}
=V_{\tw}\cap V_{{\tw}'}$.}
\end{lem}
\begin{proof}
Let $\tw=(i_1,...,i_m)$, ${\tw}'=(i'_1,...,i'_{m'})$,
and, since the case $m=m'$ follows at once from (2.7), we can and do assume
 $m' >m$. As we have also assumed that
$\tw$ and ${\tw}'$ are incomparable, we have $ (i_1,...,i_m)\ne (i'_1,...,i'_m)$.
Thus, if $Q\in K_{\tw}\cap K_{{\tw}'}$, then
$$Q\in K_{i_1,...,i_m}\cap K_{i'_1,...,i'_m}
= V_{i_1,...,i_m}\cap V_{i'_1,...,i'_m}.$$
It remains to prove that $Q\in V_{i'_1,...,i'_{m'}}$.  To see this, note that 
there exist $Q'\in K$ and $P_j\in\vo$ such that 

$$Q=\psi_{i'_1,...,i'_m}\big(\psi_{i'_{m+1},...,i'_{m'}}(Q')\big)= \psi_{i'_1,...,i'_m}(P_j).$$

Hence, by Lemma 2.1, we have $Q'=P_j$ (and $i'_{m+1}=\cdots=i'_{m'}=j$).
Thus $Q\in V_{i'_1,...,i'_{m'}}$.
\end{proof}

\begin{lem}\label{2.3}
For every $m'\ge m$ we have
$$K_{i_1,...,i_{m'}}\cap \vn {m}\subseteq V_{i_1,...,i_{m}} \eqno (2.9)$$

hence $K_{i_1,...,i_{m'}}\cap \vn {m}$ has at most one element
if $m'>m$.
 \end{lem}
\begin{proof}
Let $Q\in K_{i_1,...,i_{m'}}\cap \vn {m}$ and let $i'_1,...,i'_{m}$
be such that $Q\in  V_{i'_1,...,i'_{m}}$. Then, if
$(i_1,...,i_m)= (i'_1,...,i'_{m})$ we have $Q\in V_{i_1,...,i_{m}}$.
If, on the contrary,  $(i_1,...,i_m)\ne (i'_1,...,i'_{m})$, we have

$$Q\in K_{i_1,...,i_{m}}\cap K_{i'_1,...,i'_{m}}
\subseteq V_{i_1,...,i_{m}}, $$

and (2.9) is proved. Since every point in $K_{i_1,...,i_{m'}}\cap \vn {m}$
has the form $\psi_{i_1,..,i_{m'}}(Q)=\psi_{i_1,...,i_{m}}(P_{j})$ 
with $P_{j}\in\vo$, $Q\in K$,  we  have 
$\psi_{i_{m+1},...,i_{m'}}(Q)$ $=P_{j}$
thus by Lemma 2.1, $j=i_{m+1}=\cdots= i_{m'}$, $Q=P_j=P_{i_{m+1}}$.
 \end{proof}

 \begin{cor}\label{2.4}
 The set $V_w\cap V_{w'}$ has at most one point, whenever $w,w'\in W^*$,
 $w\ne w'$. 
 \end{cor}
\begin{proof}
Let $m=|w|$, $m'=|w'|$.
If $m=m'$ this follows from (2.7$'$). If for example $m'> m$,
let $w=(i_1,...,i_m)$, $w'=(i'_1,...,i'_{m'})$. Then
$V_w\cap V_{w'}\subseteq V_{i'_1,...,i'_{m'}}\cap \vn{m}$
and we conclude by Lemma 2.3.
\end{proof}

 \begin{cor}\label{2.5}
 We have $\psi_i(\vn{m+1}\setminus \vn{m})\subseteq
\vn{m+2}\setminus \vn{m+1}$ for every $m\in\bna$ and every $i=1,...,k$.
\end{cor}
\begin{proof}
Let $Q\in \vn{m+1}\setminus \vn{m}$, namely
$Q=\psi_{i_1,...,i_{m+1}}(P_j)$. Then, clearly, $\psi_i(Q)\in \vn{m+2}$.
Also, if $\psi_i(Q)\in \vn{m+1}$, then
$\psi_i(Q)\in K_{i,i_1,...,i_m}\cap\vn {m+1}\subseteq V_{i,i_1,...,i_m}$, by 
Lemma 2.3, and $\psi_i(Q)=\psi_i\big(\psi_{i_1,...,i_m}(P_{j'})\big)$
for some $j'=1,...,N$,
hence $Q=\psi_{i_1,...,i_m}(P_{j'})\in \vn m$, a contradiction.
\end{proof}

In the sequel we will often use with no mention
the following  simple consequence of (2.7).

\begin{lem}\label{2.6}

i) $\vn{\infty}\cap K_{\tw}=\vn{\infty}_{\tw} \quad\forall\, \tw\in W^*. $

ii) $\vn{\infty}_{\tw}$ is dense in $K_{\tw}$ for every
$\tw\in W^*. $

\end{lem}
\begin{proof}
i) To prove the $\subseteq$ part, note that if $Q\in \vn{\infty}\cap K_{\tw}$
then $Q\in \vn{\infty}_{{\tw}'}$ for some ${\tw}'\in W^*$ with 
$|{\tw}'|=|\tw|$, thus, if
${\tw}'\ne \tw$ by (2.7) we have $Q\in V_{\tw}\subseteq
\vn{\infty}_{\tw}$. The $\supseteq$ part is trivial.

\qua ii) Let $Q\in K_{\tw}$. By (2.2), $Q=\pi\big(\sigma_{\tw}(w')\big)$ for some
$w'\in W$. Next, note that

$$\sigma_{\tw}\circ\sigma_{w'_{(m)} }(1^{(\infty)})\llim\limits_{m\to +\infty}
\sigma_{\tw}(w')\eqno (2.10)$$

In fact, putting $\bar m=|\tw|$, we have 
$$\big(\sigma_{\tw}\circ\sigma_{w'_{(m)} }
(1^{(\infty)})\big)_{(\bar m+m)}=\big(\sigma_{\tw}(w')\big)_{(\bar m+m)},$$
hence (2.10) follows from the definition of the topology
on $W$. It follows from (2.10) that
$\disp{\pi\big(\sigma_{\tw}\circ\sigma_{w'_{(m)} }(1^{(\infty)})\big)
\llim\limits_{m\to +\infty}\pi\big(\sigma_{\tw}(w')\big)=Q.}$	
On the other hand, 
$$\pi\big(\sigma_{\tw}\circ\sigma_{w'_{(m)} }(1^{(\infty)})\big)=
\psi_{\tw}\circ\psi_{w'_{(m)} }\big(\pi(1^{(\infty)})\big),\eqno (2.11)$$
and since $\psi_1(P_1)=P_1$, we have $\psi_{1^{(m)}}(P_1)=P_1$, hence by
(2.3), $\pi(1^{(\infty)})=P_1$, thus, in view of (2.11), 
$\pi\big(\sigma_{\tw}\circ\sigma_{w'_{(m)} }(1^{(\infty)})\big)\in 
\vn{\infty}_{\tw}$.
\end{proof}

Finally, we require that the fractal is {\it connected}.
By this we mean that for every $Q,Q'\in\von$ there exist
$Q_0,...,Q_n\in\von$ such that $Q_0=Q$, $Q_n=Q'$, and
for every $h=1,...,n$ there exists $i(h)=1,...,k$ such that
$Q_{h-1},Q_h\in V_{i(h)}$. Note that for example
the Cantor set is not connected.
{\it From now on, all fractals are meant to have all the properties required in this Section.} 

\qua 
I now introduce the problems discussed in this paper.
Following Introduction, we say  that
$\balpha:=(\alpha_1,...,\alpha_k)$ 
is a {\it polyratio} if $\alpha_i\in ]0,1[$ for every $i=1,...,k$. 
Here we put 
$\alpha_{\tw}:=\alpha_{i_1}\cdots \alpha_{i_m}$, $\alpha_{\empti}=1$.
Put $\bar\alpha_{min}=\min\{\alpha_i\}$, $\bar\alpha_{max}=\max\{\alpha_i\}$.
Given a polyratio $\balpha$ as above we say
 that a distance $d$ on $K$ is 
  $\balpha$-{\it self-similar} if
  the maps $\psi_i$, $i=1,...,k$,  are $\alpha_i$-Lipschitz.
  If an $\balpha$-self-similar distance exists on $K$ we say that
   $\alpha$ is a {\it metric polyratio} (on $K$).  
   We  say that a distance $d$ is  an \asd\ if
  there exist  positive constants
  $c_{1,\alpha}, c_{2,\alpha}$ such that for every 
  $i_1,...,i_m=1,...,k$ (1.3) holds. 
  We say that $\balpha$ 
   is an {\it asymptotic metric polyratio} (on $K$) or shortly
{\it as.$\,$metric polyratio} if 
  there exists an \asd\ on $K$.
The problems discussed in this paper are what polyratios are metric and
what polyratios are as. metric. We will treat such problems in Section 5.
Sections 3 and 4 are devoted to introduce preparatory notions.

 \section{Graphs on the fractal}
 
 In this Section we first define a suitable graph on $\vn{\infty}$,
 and then, based on it, the notion of a path on $\vn{\infty}$.
 We define a graph on $\vn{\infty}$ putting $Q\sim Q'$ 
  for $Q,Q'\in \vn{\infty}$ if $Q\ne Q'$ and there exist
 $P,P'\in\vo$ and $\tw\in W^*$ such that $Q=\psi_{\tw}(P)$, $Q'=\psi_{\tw}(P')$.
 Put also $Q\simeq Q'$ if either $Q\sim Q'$ or $Q=Q'$.
 Given $\iota=(j_1,j_2)\in \widehat J$,  let 
 
 $$P_{\iota}:= (P_{j_1},P_{j_2}),
\quad  \psi_{\tw}(P_{\iota})=\big(\psi_{\tw}(P_{j_1}),   \psi_{\tw}(P_{j_2}) \big).$$

Let $\witi Y$ be the set of $(Q,Q')\in\vn{\infty}\times\vn{\infty}$ such that
$Q\sim Q'$.
The following function  will be useful in the sequel:
 Let $\alw:\witi Y \to\bre$ be defined as
 
 $$\alw\big(\psi_{\tw}(P),\psi_{\tw}(P')\big)=\alpha_{\tw}
 \ \ \forall P,P'\in\vo \ \forall\, \tw\in W^*.$$
 Note that, in view of (2.7$'$), such a definition is correct, 
 i.e.,  the pair $(Q,Q')$ with $Q,Q'\in \vn{\infty}$, $Q\ne Q'$ can be represented 
uniquely  as $(\psi_w(P),\psi_w(P'))$.
 The following lemma is a consequence of the assumption on the fractal
 (in particular of (2.7)).

\begin{lem}\label{3.1}
If $Q,Q'\in\vn{\infty}$, $Q\sim Q'$, and $Q\in K_{i_1,...,i_m}$, $Q'\notin 
K_{i_1,...,i_m}$, then $Q\in V_{i_1,...,i_m}$.
\end{lem}
\begin{proof}
We have $Q, Q'\in
V_{i'_1,...,i'_l}$ for some $i'_1,...,i'_l$. Let $h=\min\{l,   m\}$.
If 

$$(i'_1,...,i'_h)=(i_1,...,i_h),$$

  then

$$Q'\in K_{i_1,...,i_h}\setminus K_{i_1,...,i_{m}},$$

 thus $m>h=l$ and 

$$Q\in
V_{i_1,...,i_l}\cap K_{i_1,...,i_{m}}\subseteq V_{i_1,...,i_{m}},$$

where the inclusion holds 
since,  if 

$$Q=\psi_{i_1,...,i_l}\circ \psi_{i_{l+1},...,i_{m}}(\witi Q)=\psi_{i_1,...,i_l}(P)$$
with $\witi Q\in K$, $P\in \vo$, we have 
$\psi_{i_{l+1},...,i_{m}}(\witi Q)=P$ thus by Lemma 2.1, $\witi Q\in\vo$.

\qua If instead $(i'_1,...,i'_h)\ne (i_1,...,i_h)$, then
either $l\ge m$ and  $Q\in K_{i_1,...,i_{m}}\cap
K_{i'_1,...,i'_{m}}$  or
$l<m$ and $Q\in K_{i'_1,...,i'_{l}}$, thus, by (2.4),
$Q\in K_{i'_1,...,i'_{m}}$ for some $i'_{l+1},...,i'_{m}$. In both cases 
 $(i_1,...,i_{m})\ne (i'_1,...,i'_{m})$, thus in view of (2.7),
 $Q\in V_{i_1,...,i_{m}}$. \end{proof}

  A $\vn{\infty}$-path (or simply path) $\Pi$ is a  sequence of the form 
  
  $$(Q_{0},...,Q_n)=
 (Q_{0,\Pi},...,Q_{n(\Pi),\Pi})$$
 
   such that
 $Q_h\in \vn{\infty}$ for every $h=0,...,n$
 and, moreover, $Q_{h-1}\sim Q_h$ for every $h=1,...,n$. Thus, there exist
 $\witi w(h,\Pi)\in W^*$, $\witi \iota(h,\Pi)\in \widehat J $ such that 
  $$(Q_{h-1,\Pi},Q_{h,\Pi})=\psi_{\witi w(h,\Pi)}(P_{\witi \iota(h,\Pi)}).$$
 More generally we say that a {\it weak path} is a sequence
 $(Q_0,...,Q_n)$ such that $Q_h\in \vn{\infty}$ for every $h=0,...,n$
 and, moreover, $Q_{h-1}\simeq Q_h$ for every $h=1,...,n$. 
 
 \qua Here, we say that $Q_{h,\Pi}$, $h=0,...,n(\Pi)$, are the {\it vertices} of $\Pi$,
and that the pairs $(Q_{h-1,\Pi},Q_{h,\Pi})$, $h=1,...,n(\Pi)$, are the {\it edges} of $\Pi$.
We say that $n(\Pi)+1$ is the {\it length} of $\Pi$.

\qua An  $E$-path $\Pi$ is a path whose elements belong to $E$
 whenever $E\subseteq \vn{\infty}$. In particular we will
 use the terms $\von$-paths, $\vo$-paths.  
 So, $\Pi$ is a $\von$-path if  for every $h=1,...,n(\Pi)$, 
 $Q_{h,\Pi}\ne Q_{h-1,\Pi}$
 and either 
$Q_{h,\Pi}, Q_{h-1,\Pi}$ lie in a common $1$-cell, or both lie in $\vo$
(see Lemma 2.3). 
 
 \qua A path $\Pi$ is {\it strong} if $\witi w(h,\Pi)\ne\empti$ for every $h=1,...,n(\Pi)$,
 in other words if two consecutive vertices of the path are not both
 in $\vo$.
 We say  that $\Pi$ is  a $\iota$-path if 
 $P_{\iota}=(Q_{0,\Pi},Q_{n(\Pi),\Pi})$.
If $E\subseteq \vn{\infty}$, we say that $\Pi$ is  a $(\iota,E)$-path if it is both 
a $E$-path and a $\iota$-path.
 We say that a path is a strong $\iota$-path if it is 
 both a strong  path and $\iota$-path.
 We say that $\Pi$ connects 
 $Q$ to $Q'$ if $Q_{0,\Pi}=Q$ and $Q_{n(\Pi),\Pi}=Q'$.
 We say that $\Pi$ is a strict path if $Q_{h,\Pi}\ne Q_{h',\Pi}$ when $h\ne h'$.

 \qua When $\Pi$ is a path and $\tw\in W^*$, we define the path
   $\psi_{\tw}(\Pi)=\big(\psi_{\tw}(Q_{0,\Pi}),...,\psi_{\tw}(Q_{n(\Pi),\Pi})\big)$.
 If $\Pi=(Q_0,...,Q_n)$
 and $\Pi'=(Q'_0,...,Q'_n)$ with $Q_n=Q'_0$
 we put 
 
 $$\Pi\circ\Pi'=(Q_0,...,Q_n,Q'_1,...,Q'_n).$$
 
 A {\it subpath} of a path 
 $\Pi$ is a path of the form
 $$\Pi':=(Q_{n_1,\Pi},Q_{n_1+1,\Pi},...,Q_{n_2-1,\Pi},Q_{n_2,\Pi})$$ 
 
 with $0\le n_1\le n_2\le n(\Pi)$, and put $n_1=:n_1(\Pi,\Pi')$,  $n_2=:n_2(\Pi,\Pi')$.
Note the following  simple properties of the graphs.

 \begin{lem}\label{3.2}
  
  i) If $Q,Q'\in K$, $w^*\in W^*$, then $Q\sim Q'$ if and only if
  $\psi_{w^*}(Q)\sim \psi_{w^*}(Q')$.
  
  ii) If $w^*\in W^*$, then the sequence $(Q_0,...,Q_n)$ is  a path
  if and only if $\big(\psi_{w^*}(Q_0),...,\psi_{w^*}(Q_n)\big)$ is  a path.
   \end{lem}
 \begin{proof}
 i) The $\Rightarrow$ part is trivial. Prove $\Leftarrow$. Suppose
 $\psi_{w^*}(Q)\sim \psi_{w^*}(Q')$, so that there exist $w^*_1\in W^*$
 and $P_j, P_{j'}\in\vo$
 such that $\psi_{w^*}(Q)=\psi_{w^*_1}(P_j)$, 
 $\psi_{w^*}(Q')=\psi_{w^*_1}(P_{j'})$.
 Let $w^*_1:=(i_1,...,i_m)$,  $w^*:=(i'_1,...,i'_{m'})$.
 Let $\witi Q=\psi_{w^*}(Q)$,   $\witi Q'=\psi_{w^*}(Q')$. We have
 $\witi Q\ne \witi Q'$ and 
 
 $$\witi Q,\witi Q'\in K_{i'_1,...,i'_{ m'}}\cap 
 K_{i_1,...,i_{ m}}\cap \vn{m},$$
 
  thus, on one hand, by Lemma 2.3 we have $m'\le m$,
  on the other by (2.7$'$) and (2.4$''$),
  $(i_1,...,i_{ m'})=(i'_1,...,i'_{ m'})$.
  It follows that 
 $$\psi_{i'_1,...,i'_{m'}}(Q)=\psi_{i'_1,...,i'_{m'}}\big(\psi_{i_{m'+1},...,i_{m}}(P_j)\big),$$
 $$\psi_{i'_1,...,i'_{m'}}(Q')=\psi_{i'_1,...,i'_{m'}}\big(\psi_{i_{m'+1},...,i_{m}}(P_{j'})\big),$$
 
 and $Q=\psi_{i_{m'+1},...,i_{m}}(P_j)\sim \psi_{i_{m'+1},...,i_{m}}(P_{j'})=Q'$.
 Thus i) is proved, and ii) is a simple consequence of i).
 \end{proof}

 In the next lemmas we investigate some properties of the graphs connecting
 points lying in some specific subsets of $\vn{\infty}$. In particular,
 the statement of Lemma 3.3 corresponds to the intuitive idea that
 $V_{i_1,...,i_m}$ is a  sort of boundary of $K_{i_1,...,i_m}$. Hence,
 a path connecting a point of  $K_{i_1,...,i_m}$ to a point not in
 $K_{i_1,...,i_m}$ necessarily passes through $V_{i_1,...,i_m}$.

 \begin{lem}\label{3.3}
If $\Pi$ is a path connecting $Q\notin K_{i_1,...,i_m}$ to $Q'\in K_{i_1,...,i_m}$, then
there exists $\bar n\le n(\Pi)$ such that $Q_{n,\Pi}\notin K_{i_1,...,i_m}$ 
for every $n\le \bar n$ and
$Q_{\bar n+1,\Pi}\in V_{i_1,...,i_m}$.
\end{lem}
\begin{proof}
Let $\bar n$ be the maximum $n$ such that 
$$Q_{n',\Pi}\notin  K_{i_1,...,i_{m}}\ \ \forall n'=0,...,n.$$

Then $Q_{\bar n,\Pi}\notin K_{i_1,...,i_m}$, $Q_{\bar n+1,\Pi}\in K_{i_1,...,i_m}$,
and
$Q_{\bar n,\Pi}\sim Q_{\bar n+1,\Pi}$. Thus, by Lemma 3.1, 
$Q_{\bar n+1,\Pi}\in V_{i_1,...,i_m}$.
\end{proof}

 \begin{lem}\label{3.4}
Suppose $\Pi$ is a  path connecting $Q\in K_{i_1,...,i_m}
\setminus \vn {m}$ to $Q'\in \vn {m}$. Then there exists $\bar n$
such that $Q_{\bar n,\Pi} \in V_{i_1,...,i_{m}}$
and $Q_{n,\Pi} \in K_{i_1,...,i_{m}}$ for every $n\le \bar n$.
\end{lem}
\begin{proof}
The proof is similar to that of Lemma 3.3 but a bit more complicated.
We have $Q_{0,\Pi}=Q$, $Q_{n(\Pi),\Pi}=Q'$.
Let $\bar n$ be the maximum $n\in[0,n(\Pi)]$ such that 
$$Q_{n',\Pi}\in  K_{i_1,...,i_{m}}\ \ \forall n'=0,...,n.$$

We will prove that $Q_{\bar n,\Pi}\in  V_{i_1,...,i_{m}}$.
 If $\bar n=n(\Pi)$,
then, in view of Lemma 2.3, 
$Q_{\bar n,\Pi}\in K_{i_1,...,i_{m}}\cap \vn m\subseteq V_{i_1,...,i_{m}}.$ 

\qua  If $\bar n< n(\Pi)$ then we have $Q_{\bar n,\Pi}\sim Q_{\bar n+1,\Pi}$
by the definition of a path, and 
$Q_{\bar n,\Pi}\in  K_{i_1,...,i_{m}},$
$Q_{\bar n+1,\Pi}\notin K_{i_1,...,i_{m}},$
by the definition of $\bar n$. Thus, we have  
$Q_{\bar n,\Pi} \in V_{i_1,...,i_{m}}$ by Lemma 3.1.
\end{proof}

Note that, given a path $(Q_0,...,Q_n)$  connecting $Q_0$ to
 $Q_n$, we can associate to it a sort of reverse path connecting $Q_n$ to
 $Q_0$, that is $(Q_n,...,Q_0)$. Thus, we can use Lemmas 3.3 and 3.4
 (and other similar lemmas) also in the reverse direction. This is what we will
 do in the proof  of Lemma 3.5.

\begin{lem}\label{3.5}
If $\Pi$ is a path connecting two points in $K_{i_1,...,i_m}$
and $\Pi$ is not entirely contained in  $K_{i_1,...,i_m}$, then there exist $l<m$ 
 and a
subpath of $\Pi$ of length greater than $2$,
connecting two points in $K_{i_1,...,i_{l+1}}$, entirely contained
in $K_{i_1,...,i_{l}}$.
\end{lem}
\begin{proof}
Since every point of $\Pi$ lies in $K=K_{\empti}$, by our hypothesis there exists
a  natural $l< m$  (possibly $0$) such that
there exists a point  $Q_{\bar n,\Pi}\in K_{i_1,...,i_{l}}\setminus  K_{i_1,...,i_{l+1}}$.
We can assume $l$ is the minimum natural number satisfying   such a property.
By Lemma 3.3 there exist $n_1$ and $n_2$ with $n_1<\bar n< n_2$ such that
$Q_{n,\Pi}\notin K_{i_1,...,i_{l+1}}$ for every $n=n_1+1,...,n_2-1$
and $Q_{n_1,\Pi}, Q_{n_2,\Pi}\in V_{i_1,...,i_{l+1}}\subseteq K_{i_1,...,i_{l}}$.
 But by the definition of $l$ we
then have $Q_{n,\Pi}\in K_{i_1,...,i_{l}}$ for every $n=n_1+1,...,n_2-1$.
\end{proof}

 \qua Suppose we are given a path $\Pi$.
 We now describe a way to construct a longer path  
 by inserting new paths between any pair of consecutive vertices of $\Pi$,
 and in Lemma 3.6 we will prove that any path connecting
 two given points $Q$ and $Q'$ 
 can be obtained by repeating this process
 starting from a path connecting $Q$ and $Q'$
  lying in the same level $\vn m$ 
 of the fractal as the one where $Q$ and $Q'$ stay.
 
\qua  More formally, let  $\Gamma$ be the set of the functions 
$\gamma$  from $\widehat J$ to
 the set of the strict paths such that $\gamma(\iota)$ is a $(\iota,\von)$-path.
  Let  $\bar \Gamma$ be the set of the functions $\bar\gamma$  from $\jet$ to
 the set of the strict paths such that $\bar\gamma(\iota, w)$ is a $(\iota,\von)$-path. 
  For $\bar\gamma\in\bar\Gamma$, let
 
  $$\bar D(\bar \gamma)\big(\psi_w(P_{j_1}),  \psi_w(P_{j_2})\big)=
 \psi_w\Big(\bar\gamma\big((j_1,j_2),w\big) \Big),$$
 $$\bar D(\bar\gamma)\big(Q_0,...,Q_n\big)=
 \bar D(\bar\gamma)\big(Q_0,Q_1\big)\circ\cdots\circ
\bar D(\bar\gamma)\big(Q_{n-1},Q_n\big),$$

 when $\Pi:=(Q_0,...,Q_n)$ is a  path. When this happens, clearly, 
  $\bar D(\bar\gamma)\big(\Pi\big)$ is a path as well.
  We will say that $\bar D(\bar\gamma)\big(\Pi\big)$ is the $\bar\gamma$-{\it insertion}
  of $\Pi$.
Moreover, note that if $Q_{0,\Pi}=Q_{0,\bar D(\bar\gamma)(\Pi)}$,
$Q_{n(\Pi),\Pi}=Q_{n\big(\bar D(\bar\gamma)(\Pi)\big),\bar D(\bar\gamma)(\Pi)}$,
in other words, $\Pi$ and $\bar D(\bar\gamma)(\Pi)$ have the same end-points.

 \begin{lem}\label{3.6}
If $\Pi'$ is a  strict path connecting two points $Q$ and $Q'$
of $\vn m$, then there exist a  strict path $\Pi$ lying in $\vn m$ connecting
$Q$ and $Q'$ and $\bar\gamma_1,...,\bar\gamma_r\in\bar\Gamma$ such that

$$\Pi'=\bar D(\bar\gamma_1)\circ\cdots\circ \bar D(\bar\gamma_r)(\Pi).$$
\end{lem}
\begin{proof}
We proceed by induction  on the number of points $s(\Pi')$ of $\Pi'$ not lying
in $\vn m$. The lemma is trivial if $s(\Pi')=0$. Suppose $s(\Pi')>0$
and let $Q_{\bar n,\Pi'}\notin \vn m$. 
Thus $Q_{\bar n,\Pi'}\in \vn {\bar m}\setminus
\vn {\bar m-1}$ for some $\bar m>m$. 
We can assume that $\bar m$ is the maximum index with such  a property, in other words

$$Q_{n,\Pi'}\in \vn{\bar m} \ \ \forall n=0,...,n(\Pi').\eqno (3.1)$$

Since $Q,Q'\in\vn m\subseteq
\vn {\bar m-1}$, we have
$0<\bar n< n(\Pi')$. Moreover,  there exist $\bar i_1,...,\bar i_{\bar m-1}$ such that 
$Q_{\bar n,\Pi'}\in K_{\bar i_1,...,\bar i_{\bar m-1}}$, 
thus $Q_{\bar n,\Pi'}\in K_{\bar i_1,...,\bar i_{\bar m-1}}\setminus \vn {\bar m-1}$.
By Lemma 3.4 there exist $n_1, n_2$ such that $n_1<\bar n< n_2$,
and $Q_{n_1,\Pi'}, Q_{n_2,\Pi'}\in V_{\bar i_1,...,\bar i_{\bar m-1}}$
and 
$$Q_{n,\Pi'}\in K_{\bar i_1,...,\bar i_{\bar m-1}} \  \ \text{if}\ n_1\le n\le n_2.\eqno (3.2)$$ 
Let 

$$\Pi''=\big(Q_{0,\Pi'},....,Q_{n_1,\Pi'}, Q_{n_2,\Pi'},...,Q_{n(\Pi'),\Pi'}\big).$$

Then, $\Pi''$ is a  strict path and $s(\Pi'')<s(\Pi')$. By the inductive
hypothesis there exist  a  strict path $\Pi$ lying in $\vn m$ connecting
$Q$ and $Q'$ and
$\bar\gamma_1,...,\bar\gamma_{r}\in\bar\Gamma$ such that

$$\Pi''=\bar D(\bar\gamma_1)\circ\cdots\circ \bar D(\bar\gamma_{r})(\Pi).$$

By (3.1) and (3.2) and (2.4$'$), for every 
$n$ such that
$n_1\le n\le n_2$ there exists   $\bar i_{\bar m}=1,...,k$ (depending on $n$) such that
we have 

$$Q_{n,\Pi'}\in K_{\bar i_1,...,\bar i_{\bar m-1},\bar i_{\bar m}}\cap 
\vn{\bar m}\subseteq V_{\bar i_1,...,\bar i_{\bar m-1},\bar i_{\bar m}}
\subseteq {\von}_{\bar i_1,...,\bar i_{\bar m-1}}$$

 where  the first inclusion follows from Lemma 2.3. Hence, in view of Lemma 3.2 ii),

$$\bar\Pi:=\big(\psi_{\bar i_1,...,\bar i_{\bar m-1}}^{-1}(Q_{n_1,\Pi'}),...,
\psi^{-1}_{\bar i_1,...,\bar i_{\bar m-1}}(Q_{n_2,\Pi'})\big)$$ 

 is a strict $(\bar\iota,\von)$-path for some $\bar\iota\in\widehat J$.  Therefore,
 $(Q_{n_1,\Pi'},Q_{n_2,\Pi'})=
\psi_{\bar i_1,...,\bar i_{\bar m-1}}(P_{\bar \iota})$, 
 and let $\bar\gamma\in\bar\Gamma$ be defined
by 

$$
\bar\gamma(\iota,w^*)=\begin{cases} \bar\Pi &\ \text{if} \ \,  
(\iota,w^*)=(\bar\iota,\bar i_1,...,\bar i_{\bar m-1})\\
P_{\iota} & \ \text{otherwise}.\end{cases} 
$$

Note that 

$$\bar D(\bar\gamma)(Q_{n_1,\Pi'}, Q_{n_2,\Pi'})=
\bar D(\bar\gamma) \big(\psi_{\bar i_1,...,\bar i_{\bar m-1}}(P_{\bar \iota})\big)=$$
$$\psi_{\bar i_1,...,\bar i_{\bar m-1}}\big(\bar\gamma
(\bar\iota, \bar i_1...\bar i_{\bar m-1})\big)=$$
$$\psi_{\bar i_1,...,\bar i_{\bar m-1}}(\bar\Pi)=(Q_{n_1,\Pi'},...,Q_{n_2,\Pi'})$$

so that

$$D\big(\bar \gamma)(Q_{h,\Pi''},Q_{h+1,\Pi''})=$$
$$\begin{cases} \bar\gamma(Q_{n_1,\Pi'}, Q_{n_2,\Pi'})=(Q_{n_1,\Pi'},...,Q_{n_2,\Pi'})
&\text{if}\ h=n_1\\  (Q_{h,\Pi''},Q_{h+1,\Pi''}) &\text  {otherwise}
\end{cases}$$

and, putting $Q_n$ short for $Q_{n,\Pi'}$, we have

$$\bar D(\bar\gamma) (\Pi'')=$$
$$\bar D(\bar\gamma) (Q_{0},Q_{1})\circ
\cdots\circ \bar D(\bar\gamma) (Q_{n_1},Q_{n_2})
\circ\cdots\circ
\bar D(\bar\gamma) (Q_{n(\Pi')-1},Q_{n(\Pi')})$$
$$=(Q_{0},Q_{1},...,Q_{n_1},..., 
Q_{n_2},...,Q_{n(\Pi')-1},Q_{n(\Pi')})=\Pi',$$

so that

$$\bar D(\bar\gamma)\circ
\bar D(\bar\gamma_1)\circ\cdots\circ 
\bar D(\bar\gamma_{r})(\Pi)=\bar D(\bar\gamma) (\Pi'')=\Pi'.
$$
\end{proof}

\section{Linear operators related to the paths} 
 
  In this Section we associate some special
   linear operators (or equivalently matrices)
  to  every element of $\Gamma$ of or $\bar\Gamma$.
   Basically,
   we interpret  the edges of the paths as 
   the basis of a  linear space. Recall
   that  we have defined the function $\alw$  on
   $\witi Y$ with real values. Then, 
   we associate to the given path the vector
   having as the $(Q,Q')$-component the value $\alw(Q,Q')$.
  More precisely,
 let $e_{\iota}$ be the elements of the canonical basis in $\bre^{\widehat J}$, that is
 $\big(e_{\iota}\big)_{\iota'}=\delta_{\iota,\iota'}$.
 For every $\gamma\in \Gamma$ we define a linear operator 
 $T\pal_{\gamma}$
 from $\bre^{\widehat J}$ to $\bre^{\widehat J}$ by
 
 $$T\pal_{\gamma}(e_{\iota})=\sum\limits_{\iota'\in \widehat J}
\witi\alpha_{\gamma,\iota,\iota'} e_{\iota'}, 
\quad \witi\alpha_{\gamma,\iota,\iota'}:=
 \sum\limits_{h:\witi \iota(h,\gamma(\iota))=\iota'}
 \alpha_{\witi w(h,\gamma(\iota))}$$
 
 or in other words
 
 $$T\pal_{\gamma}(e_{\iota})=
 \sum\limits_{h=1}^{n(\gamma(\iota))}
 \alpha_{\witi w(h,\gamma(\iota))} e_{\tilde\iota(h,\gamma(\iota))}.
 $$

 Now, we introduce similar notions related to $\bar \Gamma$ instead of to $\Gamma$. 
 The reason for which we introduce
 these notions  both in $\Gamma$ and
 in $\bar\Gamma$ is that on one hand we have 
 a closer relationship between $\bar D(\bar\gamma)$ and the linear operators 
 related to $\bar\Gamma$, on the other the linear operator
 related to $\Gamma$ are simpler to handle. However,
 for our purposes the two notions are equivalent, 
 more precisely,
the asymptotic behaviors of the composition of linear operators
are in some sense the same in both situations (Corollary 4.4).

\qua  Let $e_{\iota, \tw}$ be the elements of the canonical basis of
  $\bar Z$ where $\bar Z$ is the set
 of elements $z$ of 
 $\bre^{\jet}$ such that $z_{\iota, \tw}=0$ for almost  all $(\iota, \tw)$
 (i.e., for all $(\iota, \tw)$ but finitely many). That is,
 $\big(e_{\iota, \tw}\big)_{\iota', {\tw}'}=\delta_{(\iota, \tw),(\iota', {\tw}')}$.
 Let $\bar Z_m$ be the subset of $\bar Z$ of the elements $z$ of 
 $\bre^{\jet}$ such that $z_{\iota, \tw}=0$ for every 
 $\tw\notin \witi W_{m}$.
 We also set $e_{\iota}:=e_{\iota, \empti}$.
 For every $\bar\gamma\in \bar\Gamma$ we define a linear operator 
 $\bar T\pal_{\bar\gamma}$
 from $\bar Z$ into itself   by
 
 $$\bar T\pal_{\bar\gamma}(e_{\iota,\tw})=
 \sum\limits_{h=1}^{n(\bar\gamma(\iota,\tw))} 
  \alpha_{\witi w(h,\bar\gamma(\iota,\tw))} 
  e_{ \witi \iota\big(h,\bar\gamma(\iota,\tw)\big),
  \tw\witi w\big(h,\bar\gamma(\iota,\tw)\big) }.
 $$
 
 Note that for every
 $m\in\bna$, $\bar T\bar\pal_{\bar\gamma}$ maps $\bar Z_m$ into
 $\bar Z_{m+1}$.
 Let $H:\bre^{\widehat J}\to \bre$
 $\bar H:\bar Z \to \bre$
  be defined by 
 $$H(x)=\sum\limits_{\iota\in \widehat J} x_{\iota}, \quad
 \quad \bar H(x)=\sum\limits_{(\iota,\tw)\in \widehat J\times W^*} x_{\iota,\tw}.$$ 
 
 Let $f:\witi Y\to\bre$. We now define the {\it sum operators along a path},
that is,  some kind  of sum of
 $f$ along a given path. We will use such notions in the sequel specially
 when $f=\alw$, but also other $f$ will be considered in Section 5.
 Namely,  if
 $\Pi$ is a  path, we define 
 
 $$\witi\Sigma_{\Pi}(f)=\sum\limits_{h=1}^{n(\Pi)} f(Q_{h-1,\Pi},Q_{h,\Pi})
 e_{\witi \iota(h,\Pi), \witi w(h,\Pi)}\in\bar Z,$$
 
 $$\widehat\Sigma_{\Pi}(f)=\sum\limits_{h=1}^{n(\Pi)} f(Q_{h-1,\Pi},Q_{h,\Pi})
 e_{\witi \iota(h,\Pi)}\in\bre^{\widehat J},$$
 
 $$\Sigma_{\Pi}(f)=\sum\limits_{h=1}^{n(\Pi)} f(Q_{h-1,\Pi},Q_{h,\Pi}).$$

We will occasionally use such definitions (more specifically, the definition
of $\Sigma_{\Pi}(f)$) also when $\Pi$ is a weal path, with the convention $f(Q,Q)=0$. 
We will use the following simple properties in the sequel without mention
for paths $\Pi$, $\Pi'$.
$$\Sigma_{\Pi}(f)=\bar H\Big(\witi\Sigma_{\Pi}(f)\Big)=H\Big(\widehat\Sigma_{\Pi}(f)\Big).
$$

 $$\witi\Sigma_{\Pi\circ\Pi'}(f)=\witi\Sigma_{\Pi}(f)+\witi\Sigma_{\Pi'}(f),
 \quad \widehat\Sigma_{\Pi\circ\Pi'}(f)=\widehat\Sigma_{\Pi}(f)+\widehat\Sigma_{\Pi'}(f).$$

 Recall that, given $\iota=(j_1,j_2)\in\widehat J$, the symbol
 $P_{\iota}$ denotes the pair $(P_{j_1}, P_{j_2})$, which, of course can be interpreted
 as a path of length  $1$. We immediately have:

 $$\widehat\Sigma_{P_{\iota}}(\alw)= e_{\iota},   
 \quad    \witi\Sigma_{P_{\iota}}(\alw)= e_{\iota},    \eqno(4.1)    $$

\begin{rem}  \label{4.1}
\text{\rm When$\,\Pi\,$is a path, the following formulas are immediate}
\text{\rm  consequences of the definition
of} $\witi w(h,\Pi)$.
$$\alw(Q_{h-1,\Pi},Q_{h,\Pi})=\alpha_{\witi w(h,\Pi)}, \eqno(4.2)$$
$$\Sigma_{\Pi}(\alw)=\sum\limits_{h=1}^{n(\Pi)}  \alpha_{\witi w(h,\Pi)}.\eqno (4.2')$$
\text{\rm Note that (4.2) and (4.2$'$) are also valid when $\,\Pi\,$ is a weak path
$\,$if we} 

\text{\rm  use the convention $\alpha_{\witi w(h,\Pi)}=0$ when $Q_{h-1}=Q_h$.}
\end{rem}

 For every $\gamma\in\Gamma$ let  $I(\gamma)\in \bar\Gamma$ defined by 
$I(\gamma)(\iota,\tw)=\gamma(\iota)$.  Let $D(\gamma)=\bar D\big(I(\gamma)\big)$.
The following properties will be useful in the sequel.
For every $\iota,\iota'\in\widehat J$ and every $\tw\in W^*$ we have

$$\big(T\pal_{\gamma}(e_{\iota})\big)_{\iota'}=
\witi\alpha_{\gamma,\iota,\iota'}=\sum\limits_{{\tw}'\in W^*} 
\big(\bar T\pal_{I(\gamma)}(e_{\iota,w^*})\big)_{\iota', {\tw}'}\, , \eqno (4.3)$$
$$H\big(T\pal_{\gamma}(e_{\iota})\big)=\sum\limits_{h=1}^{n(\gamma(\iota))}
 \alpha_{\witi w(h,\gamma(\iota))}=\bar H\big(\bar T\pal_{I(\gamma)}(e_{\iota,w^*})\big).
\eqno (4.4)$$

Next lemma shows that the linear operators $T\pal_{\gamma}$
 and $\bar T\pal_{\bar\gamma}$ allows us to evaluate the sum operators
 along insertion paths.

 \begin{lem}\label{4.2}
 For every  path $\Pi$ 
 we have
 
 i) $\witi\Sigma_{\bar D(\bar\gamma_1)\circ\cdots\circ \bar
 D(\bar\gamma_n)(\Pi)}(\alw)=
 \bar T\pal_{\bar\gamma_1}\circ\cdots\circ \bar T\pal_{\bar\gamma_n}
 \big(\witi\Sigma_{\Pi}(\alw)\big)\ \ \forall\, \bar\gamma_1,...,\bar\gamma_n\in\bar\Gamma.$
 
 ii) $\widehat\Sigma_{D(\gamma_1)\circ\cdots\circ D(\gamma_n)(\Pi)}(\alw)=
  T\pal_{\gamma_1}\circ\cdots\circ  T\pal_{\gamma_n}
 \big(\widehat\Sigma_{\Pi}(\alw)\big)
 \ \ \forall\,  \gamma_1,...,\gamma_n\in \Gamma.$
 
 \end{lem}
 \begin{proof} it suffices to prove the Lemma when $n=1$.
  Note that for every $\iota\in \widehat J$, every
   $\tw\in W^*$ and every $\bar\gamma\in\bar\Gamma$
 we have
 
 $$\witi\Sigma_{\psi_{\tw}(\bar\gamma(\iota,\tw))} (\alw)=
\sum\limits_{h=1}^{n(\bar\gamma(\iota,\tw))}
\alpha_{\tw \witi\iota(h,\bar\gamma(\iota,\tw))}
e_{\witi w(h,\bar\gamma(\iota,\tw)), \tw\witi w(h,\bar\gamma(\iota,\tw))}$$
 $$=\alpha_{\tw}\bar T\pal_{\bar\gamma}(e_{\iota,\tw}).$$

Let $\bar\gamma\in\bar\Gamma$. We thus have
 
 $$\witi\Sigma_{\bar D(\bar\gamma)(\Pi)}(\alw)=
 \sum\limits_{h=1}^{n(\Pi)}\witi \Sigma_{\bar D(\bar\gamma)(Q_{h-1,\Pi}, Q_{h,\Pi})}
 (\alw)=$$
 $$ \sum\limits_{h=1}^{n(\Pi)} \witi\Sigma_{D (\bar\gamma)
\big (\psi_{\witi w(h,\Pi)}(P_{\witi\iota(h,\Pi)})\big)} (\alw)=$$
$$ \sum\limits_{h=1}^{n(\Pi)} \witi\Sigma_{\psi_{\witi w(h,\Pi)}
\Big(\bar\gamma\big(\witi\iota(h,\Pi),
\witi w(h,\Pi)\big)\Big)} (\alw)=$$ 
$$ \sum\limits_{h=1}^{n(\Pi)}
\alpha_{\witi w(h,\Pi)} \bar T\pal_{\bar\gamma} (e_{ \witi\iota(h,\Pi),
\witi w(h,\Pi)  })=$$
$$\bar T\pal_{\bar\gamma}\Big(\sum\limits_{h=1}^{n(\Pi)}
\alpha_{\witi w(h,\Pi)}  e_{ \witi\iota(h,\Pi),
\witi w(h,\Pi)  }\Big)=\bar T\pal_{\bar\gamma}\big(\witi \Sigma_{\Pi}(\alw)\big).
$$

This proves i), and ii) can be proved similarly.
\end{proof}

 We are now going to introduce a notion for polyratios 
 which will turn out to be equivalent 
 to those of metric and of as. metric.

 {\it \quad We say that the polyratio $\alpha$ is $(T,\Gamma)$-uniformly positive
 (or short  $(T,\Gamma)$-u.p.) if
 there exists $c_{3,\alpha}>0$ such that
for every $n\in\bna$, every  $\gamma_1,...,\gamma_n\in\Gamma$  
and every $\iota\in \widehat J$, we have

$$H\big(T\pal_{\gamma_1}\circ\cdots \circ T\pal_{\gamma_n}(e_{\iota})\big)
\ge c_{3,\alpha}. \eqno (4.5)$$}

 {\it \quad We say that the polyratio
 $\alpha$ is $(\bar T,\bar\Gamma)$-uniformly positive
 (or short  $(\bar T,\bar\Gamma)$-u.p.) if
 there exists $c_{3,\alpha}>0$ such that
for every $m\in\bna$, every  $\bar\gamma_1,...,\bar\gamma_m
\in\bar\Gamma$  
and every $\iota\in \widehat J$, we have

$$\bar H\big(\bar T\pal_{\bar\gamma_1}\circ\cdots \circ 
\bar T\pal_{\bar\gamma_m}(e_{\iota})\big)
\ge c_{3,\alpha}. \eqno (4.5')$$}

Note that, since $\widehat J$ and $\von$ are finite sets, and the elements
of $\Gamma$ are by definition {\it strict} $\von$-paths, then $\Gamma$ is a finite set
as well. Thus, as hinted in Introduction, the notion
of $(T,\Gamma)$-u.p. is related to that of joint spectral radius
and to that of joint spectral subradius of a finite set of matrices. On the contrary,
$\bar\Gamma$ is an infinite set, thus the property of being
$(\bar T,\bar \Gamma)$-u.p. is more complicated to verify. However, as we will
see in   Corollary 4.4, the two properties are equivalent. We need a Lemma.

\begin{lem}\label{4.3}
i) For every $\gamma_1,...,\gamma_{m+1}\in \Gamma$,
we have

$$ H\Big(T\pal_{\gamma_{1} }\circ\cdots\circ T\pal_{\gamma_{m+1}}(e_{\iota})\Big)=
\sum\limits_{\iota'\in \widehat J}
T\pal_{\gamma_{m+1} } (e_{\iota})_{\iota'} H
\Big( T\pal_{\gamma_{1} }\circ\cdots\circ T\pal_{\gamma_{m}}(e_{\iota'})\Big).
$$

ii) For every $\bar\gamma_1,...,\bar\gamma_{n+1}\in \bar\Gamma$,
we have

$$ \bar H\Big(\bar T\pal_{\bar \gamma_{1} }\circ\cdots\circ 
\bar T\pal_{\bar\gamma_{m+1}}(e_{\iota,\tw})\Big)=$$
$$\sum\limits_{(\iota',{\tw}') \in \jet}
\bar T\pal_{\bar \gamma_{m+1} } (e_{\iota, \tw})_{\iota',{\tw}'} \bar H
\Big( \bar T\pal_{\bar \gamma_{1} }\circ\cdots\circ 
\bar T\pal_{\bar \gamma_{m}}(e_{\iota',{\tw}'})\Big)=$$
$$\sum\limits_{h=1}^{n(\bar\gamma_{m+1}(\iota,\tw))}
 \witi \alpha_h
  \bar H\Big( \bar T\pal_{\bar \gamma_{1} }\circ\cdots\circ 
\bar T\pal_{\bar \gamma_{m}}(e_{ \witi \iota\big(h,\bar\gamma_{m+1}(\iota,\tw)\big),
  \tw\witi w\big(h,\bar\gamma_{m+1}(\iota,\tw)\big) })\Big),$$
$$\witi \alpha_h:= \alpha_{\witi w(h,\bar\gamma_{m+1}(\iota,\tw))}. $$

iii) Given $ \gamma_l\in  \Gamma$,  $l=1,2,...,m$, for every
$\iota\in\widehat J$ and $\tw \in W^*$
we have

$$H\Big(T\pal_{\gamma_{1} }\circ\cdots\circ T\pal_{\gamma_{m}}(e_{\iota})\Big)=
\bar H\Big(\bar T\pal_{I(\gamma_{1}) }\circ\cdots\circ 
\bar T\pal_{I(\gamma_{m})}(e_{\iota,\tw})\Big).$$

iv) Given $\bar \gamma_l\in \bar\Gamma$, 
$l=1,2,... ,m$,  there exist 
$\gamma_l\in \Gamma$  such that 
for  every $\iota\in \widehat J$ we have

$$H\Big(T\pal_{\gamma_{1} }\circ\cdots\circ T\pal_{\gamma_{m}}(e_{\iota})\Big)\le
 \bar H\Big(\bar T\pal_{\bar \gamma_{1} }\circ\cdots\circ 
\bar T\pal_{\bar\gamma_{m}}(e_{\iota})\Big). \eqno (4.6)
$$

\end{lem}
\begin{proof}
i)  Note that

$$H\Big(T\pal_{\gamma_{1} }\circ\cdots\circ T\pal_{\gamma_{m+1}}(e_{\iota})\Big)=$$
$$\sum\limits_{\iota''\in \widehat J}  \Big(\big(T\pal_{\gamma_{1} }\circ
\cdots \circ T\pal_{\gamma_m}\big)
\circ T\pal_{\gamma_{m+1}}(e_{\iota})\Big)_{\iota''}=$$
$$\sum\limits_{\iota''\in \widehat J}  \Big( T\pal_{\gamma_{1} }
\circ\cdots \circ T\pal_{\gamma_m}
 \big(\sum\limits_{\iota'\in \widehat J}
T\pal_{\gamma_{m+1}}(e_{\iota})_{\iota'} e_{\iota'}\big)\Big)_{\iota''}=$$
$$\sum\limits_{\iota''\in \widehat J} 
\Big( \sum\limits_{\iota'\in \widehat J}
T\pal_{\gamma_{m+1} } (e_{\iota})_{\iota'}
T\pal_{\gamma_1}\circ\cdots\circ T\pal_{\gamma_m}(e_{\iota'} )
\Big)_{\iota''}=$$
$$\sum\limits_{\iota'\in \widehat J} T\pal_{\gamma_{m+1} } (e_{\iota})_{\iota'}
 \Big(\sum\limits_{\iota''\in \widehat J}
T\pal_{\gamma_1}\circ\cdots\circ T\pal_{\gamma_m}(e_{\iota'} )_{\iota''}\Big)=$$
$$\sum\limits_{\iota'\in \widehat J}
T\pal_{\gamma_{m+1} } (e_{\iota})_{\iota'} H
\Big( T\pal_{\gamma_{1} }\circ\cdots\circ T\pal_{\gamma_{m}}(e_{\iota'})\Big)
$$

and i) is proved. ii) the proof of the first equality 
is very similar and is omitted. The second equality follows
from the definition of $\bar T\pal_{\bar\gamma_{m+1}}$.

iii) The case $m=1$ follows from (4.4). The general case 
follows by induction. In fact,
if iii) holds for $m$, by i),
ii),  (4.3)  and the definitions of $T(\gamma)$ 
and $\bar T(\bar\gamma)$, $H$ and $\bar H$, we have. 

$$H\Big(T\pal_{\gamma_{1} }\circ\cdots\circ T\pal_{\gamma_{m+1}}(e_{\iota})\Big)$$
$$=\sum\limits_{\iota'\in \widehat J}
T\pal_{\gamma_{m+1} } (e_{\iota})_{\iota'} H
\Big( T\pal_{\gamma_{1} }\circ\cdots\circ T\pal_{\gamma_{m}}(e_{\iota'})\Big)$$
$$=\sum\limits_{\iota'\in \widehat J}
\bigg(\sum\limits_{{\tw}'\in W^*}
\bar T\pal_{I(\gamma_{m+1}) } (e_{\iota, \tw})_{\iota',{\tw}'} \bar H
\Big( \bar T\pal_{I( \gamma_{1}) }\circ\cdots\circ 
\bar T\pal_{I(\gamma_{m})}(e_{\iota',{\tw}'})\Big)\bigg)$$
$$=\sum\limits_{(\iota', {\tw}')\in \jet}
\bar T\pal_{I(\gamma_{m+1}) } (e_{\iota, \tw})_{\iota',{\tw}'} \bar H
\Big( \bar T\pal_{I( \gamma_{1}) }\circ\cdots\circ 
\bar T\pal_{I(\gamma_{m})}(e_{\iota',{\tw}'})\Big)$$
$$=\bar H\Big(\bar T\pal_{I(\gamma_{1}) }\circ\cdots\circ 
\bar T\pal_{I(\gamma_{m+1})}(e_{\iota,\tw})\Big),
$$

and iii) holds for $m+1$.

iv)  Note that for every $\iota,\iota'\in \widehat J$, every $w^*\in \witi W_m$ and
every $\bar\gamma\in\bar\Gamma$
we have the following simple analog of (4.3)

$$\sum\limits_{{\tw}'\in \witi W^*}
\big(\bar T\pal_{\bar \gamma}(e_{\iota,w^*})\big)_{\iota', {\tw}'}=
\sum\limits_{{\tw}'\in \witi W_{m+1}}
\big(\bar T\pal_{\bar \gamma}(e_{\iota,w^*})\big)_{\iota', {\tw}'}=
\witi\alpha_{\bar\gamma,\iota,\iota',w^*}, \eqno (4.7)$$
$$\witi\alpha_{\bar\gamma,\iota,\iota',w^*}:=
 \sum\limits_{h:\witi \iota(h,\bar \gamma(\iota,w^*))=\iota'}
 \alpha_{\witi w(h,\bar \gamma(\iota,w^*))}.$$

We define $\gamma_{l}$ by induction. Suppose we have defined
$\gamma_1,...,\gamma_l$, $l<m$,  satisfying 

$$H\Big(T\pal_{\gamma_{1} }\circ\cdots\circ T\pal_{\gamma_{l}}(e_{\iota})\Big)\le
 \bar H\Big(\bar T\pal_{\bar \gamma_{1} }\circ\cdots\circ 
\bar T\pal_{\bar\gamma_{l}}(e_{\iota,w^*})\Big) \ \forall  \tw \in \witi W_{m-l}, \eqno (4.8)
$$

 and define $\gamma_{l+1}$
in this way. Let 

$$\theta_{\iota}{\tw}:=
\sum\limits_{\iota'\in \widehat J} \witi\alpha_{\bar\gamma_{l+1},\iota,\iota',w^*}
H\Big( T\pal_{\gamma_{1} }\circ\cdots\circ T\pal_{\gamma_{l}}(e_{\iota'})\Big),$$

$$\gamma_{l+1}(\iota)=\bar\gamma_{l+1}(\iota,w^*(\iota))$$

where for every given $\iota\in \widehat J$, $w^*(\iota)$ is an element
of $\witi W_{m-l-1}$ minimizing $\theta_{\iota}$, i.e., such that 

$$\theta_{\iota}\big(\tw(\iota)\big)\le
 \theta_{\iota}{\tw}\quad\forall\, \tw\in \witi W_{m-l-1}.
$$ 

Note that, in view of (4.3) and (4.7) we have
$$\sum\limits_{{\tw}'\in \witi W_{m-l}}
\bar T\pal_{\bar \gamma_{l+1} } (e_{\iota, \tw(\iota)})_{\iota',{\tw}'}=
T\pal_{ \gamma_{l+1} } (e_{\iota})_{\iota'} \eqno(4.9)
$$

for every  $\iota,\iota'\in\widehat J$.
Thus, for every $\tw\in \witi W_{m-l-1}$ we have

$$\bar H\Big(\bar T\pal_{\bar \gamma_{1} }\circ\cdots\circ 
\bar T\pal_{\bar\gamma_{l+1}}(e_{\iota,w^*})\Big)$$
$$=\sum\limits_{(\iota',{\tw}')\in \jet}
\bar T\pal_{\bar \gamma_{l+1} } (e_{\iota, \tw})_{\iota',{\tw}'} \bar H
\Big( \bar T\pal_{\bar \gamma_{1} }\circ\cdots\circ 
\bar T\pal_{\bar \gamma_{l}}(e_{\iota',{\tw}'})\Big)\eqno (\text{by\  ii)})$$
$$=\sum\limits_{\iota'\in \widehat J} \sum\limits_{{\tw}'\in \witi W_{m-l}}
\bar T\pal_{\bar \gamma_{l+1} } (e_{\iota, \tw})_{\iota',{\tw}'} \bar H
\Big( \bar T\pal_{\bar \gamma_{1} }\circ\cdots\circ 
\bar T\pal_{\bar \gamma_{l}}(e_{\iota',{\tw}'})\Big)$$
$$\ge \sum\limits_{\iota'\in \widehat J} \sum\limits_{{\tw}'\in \witi W_{m-l}}
\bar T\pal_{\bar \gamma_{l+1} } (e_{\iota, \tw})_{\iota',{\tw}'}   H
\Big(  T\pal_{\gamma_{1} }\circ\cdots\circ 
T\pal_{\gamma_{l}}(e_{\iota'})\Big) \eqno (\text{by\  (4.8)})$$
$$=\theta_{\iota}(\tw)
\ge \theta_{\iota}\big(\tw(\iota)\big)$$
$$=\sum\limits_{\iota'\in \widehat J} \sum\limits_{{\tw}'\in \witi W_{m-l}}
\bar T\pal_{\bar \gamma_{l+1} } (e_{\iota, \tw(\iota)})_{\iota',{\tw}'}   H
\Big(  T\pal_{\gamma_{1} }\circ\cdots\circ 
T\pal_{\gamma_{l}}(e_{\iota'})\Big)$$
$$=\sum\limits_{\iota'\in \widehat J}
  T\pal_{ \gamma_{l+1} } (e_{\iota})_{\iota'}   H
\Big(  T\pal_{\gamma_{1} }\circ\cdots\circ 
T\pal_{\gamma_{l}}(e_{\iota'})\Big)\eqno (\text{by\  (4.9)})$$
$$=H\Big(T\pal_ { \gamma_{1} }\circ\cdots\circ 
 T\pal_{\gamma_{l+1}}(e_{\iota})\Big) \eqno (\text{by\  i)})
$$

and (4.8) holds for $l+1$, and the inductive step is completed.
Now, iv) follows from the case $l=m$.
\end{proof}

\begin{cor}\label{4.4}
$\alpha$ is $(T,\Gamma)$-u.p. if and only if it is $(\bar T,\bar \Gamma)$-u.p.
\end{cor}
\begin{proof} This immediately follows from Lemma 4.3 iii) and iv).
\end{proof}

\section{Distances on the fractal}
 
 In this Section, we  prove that for a  polyratio $\balpha$
 the following are equivalent
 
 \qua i) $\balpha$ is a metric poluyratio
 
 \qua ii) $\balpha$ ia an as. metric polyratio
 
 \qua   iii) $\balpha$ is $(T,\Gamma)$-u.p.

 In the sequel we will always use
 implicitly  Corollary 4.4. First, we prove that both i) and ii) imply iii).

\begin{lem}\label{5.1} 
If  $\alpha$ is either a metric polyratio or an as. metric polyratio
 on $K$, then  $\alpha$ is $(T,\Gamma)$-u.p.
 \end{lem}
\begin{proof}
If  $\alpha$ is either a metric polyratio or an as. metric polyratio,
then there exists a distance $d$ on $K$ such that
for some $c_{2,\alpha}>0$ the second inequality in (1.3) holds.
Observe that for every path $\Pi$ connecting $P_{j_1}$ to $P_{j_2}$
for every $h=1,...,n(\Pi)$  we have 

$$d(P_{j_1},P_{j_2})\le\sum\limits_{h=1}^{n(\Pi)}
d(Q_{h-1,\Pi},Q_{h,\Pi})=\Sigma_{\Pi}(d).\eqno (5.1)$$

Moreover, for some $j_1, j_2$ (depending on $h$) we have
$$d(Q_{h-1,\Pi},Q_{h,\Pi})=
d\big(\psi_{\witi w(h,\Pi)}(P_{ j_1}),\psi_{\witi w(h,\Pi)}(P_{ j_2})\big)$$
$$\le \diam_{  d} K_{\witi w(h,\Pi)}\le 
c_{2,\alpha} \alpha_{\witi w(h,\Pi)} \diam_{d}(K)$$
$$=c_{2,\alpha} \diam_{ d}(K)\, \alw(Q_{h-1,\Pi},Q_{h,\Pi}),
$$

for every $h=1,..,n(\Pi)$, thus, summing the previous inequalities, we obtain

$$\Sigma_{ \Pi}(d)=\sum\limits_{h=1}^{n(\Pi)}d(Q_{h-1,\Pi},Q_{h,\Pi})$$
$$\le c_{2,\alpha} \diam_{ d}(K)\sum\limits_{h=1}^{n(\Pi)}\alw(Q_{h-1,\Pi},Q_{h,\Pi})
= c_{2,\alpha} \diam_{ d}(K) \Sigma_{ \Pi}(\alw),$$

and, in view of (5.1),

$$d(P_{j_1},P_{j_2})\le c_{2,\alpha} \diam_{ d}(K) \Sigma_{ \Pi}(\alw).
\eqno (5.2)$$

Let now $\bar\gamma_1,...,\bar\gamma_l\in\bar\Gamma$ 
and let $\{j_1,j_2\}\in \widehat J$.  Set $\Pi:=(P_{j_1},P_{j_2})$, so that 
$\witi\Sigma_{\Pi} (\alw)=e_{\{j_1,j_2\}}$ (see 4.1). Recall
that hence also the path $\bar D(\bar\gamma_1)\circ\cdots\circ \bar
D(\bar\gamma_l)(\Pi)$
connects $P_{j_1}$ to $P_{j_2}$, so that we can use (5.2) with this path in place of
$\Pi$, and, in view also of Lemma 4.2 i), we have

$$\min\big \{ d\big(P_{h_1}, P_{h_2}\big):(h_1,h_2)
\in \widehat J\big\} \le d\big(P_{j_1}, P_{j_2}\big)\le $$
$$c_{2,\alpha} \diam_d(K)
\Sigma_{\bar D(\bar\gamma_1)\circ\cdots\circ \bar D(\bar\gamma_l)(\Pi)}(\alw)=$$
$$c_{2,\alpha} \diam_d(K)\bar H\Big( \bar T\pal_{\bar\gamma_1}
\circ\cdots\circ \bar T\pal_{\bar\gamma_l}
(e_{\{j_1,j_2\} })\Big).$$

Hence, (4.5)$'$ is satisfied with

$$c_{3,\alpha}={\min\big \{ d\big(P_{h_1}, P_{h_2}\big):(h_1,h_2)
\in \widehat J\big\}
\over c_{2,\alpha} \diam_d(K)}.$$
 
\end{proof}

Recall the definition of the pseudodistance given in \cite{4}.
If $Q, Q'\in K$ and
 $\tw_1,...,\tw_n\in W^*$,
we say that $(\tw_1,...,\tw_n)$ is a  prechain if
$K_{\tw_h}\cap K_{\tw_{h+1}}\ne\empti$ for every 
$h=1,...,n-1$. We say that   $(\tw_1,...,\tw_n)$ is  between $Q$ and
$Q'$ if $Q \in K_{\tw_1}$ and $Q'\in  K_{\tw_n}$.
We say that a prechain $(\tw_1,...,\tw_n)$
is a  chain if $\tw_1$,...,$\tw_n$ are pairwise incomparable.
Denote by $G'(Q, Q')$ the
set of   prechains  between $Q$ and $Q'$ and
by $G(Q, Q')$ the
set of   chains  between $Q$ and $Q'$. 
 Let $\balpha = (\alpha_1,...,\alpha_k)$ be a polyratio
 and let

$$A(\cC) = \sum\limits_{h=1}^n \alpha_{\tw_h},\quad \text{\rm if}\
\cC\in G'(Q, Q'),$$
$$\Dal(Q,Q'):=\inf\big\{A(\cC): \cC\in G(Q,Q')\big\}
=\inf\big\{A(\cC): \cC\in G'(Q,Q')\big\}.$$

Note that a standard argument shows that $\Dal$ is in any case
a {\it pseudodistance}, in the sense that satisfies
all properties of a  distance, except possibly for the fact
that $\Dal(Q,Q')=0$ could occur also when $Q\ne Q'$.

\begin{lem}\label{5.2}

i) For every polyratio $\balpha$, $D_{\alpha}$ is a pseudodistance, and, 
if $\Dal(Q,Q')>0$, for every $Q,Q'\in\vn{\infty}$ with $Q\ne Q'$, $\Dal$ is
 an $\balpha$-self-similar distance.

ii) If $\balpha$ is a metric polyratio, then $\Dal$
is an $\balpha$-self-similar distance, which induces the same topology
as the original distance.

\end{lem}
\begin{proof}
i) This follows from  Theorem 1.33 and Prop. 1.12 in \cite{4}.
ii) See Prop. 1.13 and 1.11 in \cite{4}. Note that, according to the definitions
in \cite{4}, which are equivalent to the definitions here, but slightly differ
from them, in the hypothesis of Prop. 1.13 of \cite{4} it is assumed 
that $d$ induces the same topology on $K$ as the original distance,
but the proof does not use this fact.
\end{proof}

Now, we are going to prove the converse of Lemma 5.1, i.e.,
if $\alpha$ is $(T,\Gamma)$-u.p., then 
$\Dal$ is at the same time an 
$\alpha$-self-similar distance and an $\alpha$-scaling distance on
$K$. 
In order to do this, we will introduce a pseudodistance $\wdal$ 
on $\vn{\infty}$ that is more strictly related to the notion of 
$(T,\Gamma)$-u.p. than $\Dal$. However, in Lemma 5.6 we will prove
that $\Dal$ and $\wdal$ are equivalent on $\vn{\infty}$.
A similar distance in discussed in \cite{4}, Theorem 1.34.
Clearly, for every path $\Pi$ and for every $\tw\in W^*$ we have

$$\Sigma_{\psi_{\tw}(\Pi)}(\alw) =\alpha_{\tw}  \Sigma_{\Pi}(\alw). \eqno (5.3)
$$

Let $\path'_{Q,Q'}$ be the set of the weak paths connecting
$Q$ to $Q'$, let
$\path_{Q,Q'}$ be the set of the paths connecting
$Q$ to $Q'$, and let $\tpath_{Q,Q'}$ 
be the set of the strict paths connecting $Q$ to $Q'$.
We define now the function  $\wdal$ on $\vn{\infty}$ by

$$\wdal(Q,Q')=\inf\big\{\Sigma_{\Pi}(\alw): \Pi\in\path_{Q,Q'}\big\}=.$$
$$\inf\big\{\Sigma_{\Pi}(\alw): \Pi\in\tpath_{Q,Q'}\big\}=
\inf\big\{\Sigma_{\Pi}(\alw): \Pi\in\path'_{Q,Q'}\big\}.$$
It can be easily proved that $\wdal$ is a pseudodistance on $\vn{\infty}$. 
The only nontrivial point for proving this, is that 
for every $Q,Q'\in \vn{\infty}$ we have $\path_{Q,Q'}\ne\empti$, so that
$\wdal(Q,Q')<+\infty$. This will follow from the next lemma.

\begin{lem}\label{5.3}
There exists a  constant $C_{1,\alpha}\ge 1$ such that 

$$ \diam_{\wdal}\big(\vn {\infty}_{\tw}\big)\le C_{1,\alpha} \alpha_{\tw}$$

 for every  $\tw\in W^*$.
 Thus, in particular, $\diam_{\wdal}\big(\vn {\infty}_{\tw}\big)$ is finite,
  for every  $ \tw\in W^*$.
\end{lem}
\begin{proof}
Since the fractal is connected,
there exists $C_{\alpha}\ge 1$ such that for every $\widehat Q, \widehat Q'\in \von$
there exists a path $\Pi$ connecting them such that 
$\Sigma_{\Pi}(\alw)\le C_{\alpha}$. 
Thus, for every $\tw\in W^*$, in view of (5.3), the path $\Pi':=\psi_{\tw}(\Pi)$ satisfies
$\Sigma_{\Pi'}(\alw)\le C_{\alpha} \alpha_{\tw}$. Since $\Pi'$ connects 
$\psi_{\tw}(\widehat Q)$
to $\psi_{\tw}(\widehat Q')$, we then have

$$\wdal\big(\psi_{\tw}(\widehat Q),\psi_{\tw}
(\widehat Q')\big)\le C_{\alpha} \alpha_{\tw}.\eqno (5.4)$$
%
 %
 
 Now, suppose $\tw=(i_1...i_m)$. If $Q\in\vn{\infty}_{\tw}$, there exist
 $i_{m+1},...,i_{m'}$ and $P\in\vo$ such that
 $Q=\psi_{i_1,...,i_m, i_{m+1},...,i_{m'}}(P)$.
 Let $\Pi= (Q_m,...,Q_{m'})$, where $Q_h=
 \psi_{i_1,...,i_m,...,i_{h}}(P)$. Then $\Pi$ is a path that connects
 $\witi Q:=\psi_{i_1,...,i_m}(P)\in V_{i_1,...,i_m}$ to $Q$.
 Therefore,
 
 $$\wdal(Q,\witi Q)\le \sum\limits_{h=m+1}^{m'}
 \wdal(Q_{h-1},Q_h)$$
$$ =\sum\limits_{h=m+1}^{m'}\wdal\Big(
\psi_{i_1,...,i_m,...,i_{h-1}}(P), \psi_{i_1,...,i_m,...,i_{h-1}}\big( \psi_{i_h}(P)\big)\Big)$$
$$\le C_{\alpha} \sum\limits_{h=m+1}^{m'}\alpha_{i_1,...,i_m,...,i_{h-1}}
\eqno\text{\text \ (by\ (5.4))}$$
$$=C _{\alpha}\alpha_{\tw} \sum\limits_{h=m+1}^{m'}
\alpha_{i_{m+1}}\cdots\alpha_{i_{h-1}}$$
$$\le  C_{\alpha} \alpha_{\tw} \sum\limits_{h=m+1}^{m'}
\bar\alpha_{max}^{h-m-1}$$
$$\le { C_{\alpha}   \over 1-\bar\alpha_{max}} \,  \alpha_{\tw}
 $$
   
     Now, given $Q,Q'\in \vn{\infty}_{\tw}$, let $\witi Q$, $\witi{Q'}$ 
     be as above. Then, we have
  
  $$\wdal(Q,Q')\le \wdal(Q,\witi Q)+\wdal(\witi Q,\witi {Q'})+
  \wdal(\witi {Q'},Q')$$
  $$\le \Big(C+2 {  C_{\alpha} 
   \over 1-\bar\alpha_{max}}\Big) \alpha_{\tw}, $$
   
   thus the Lemma is proved with $C_{1,\alpha}=C_{\alpha}+2 {  C_{\alpha} 
   \over 1-\bar\alpha_{max}}$.
 \end{proof}

\begin{lem}\label{5.4}
If $\alpha$ is $(T,\Gamma)$ u.p, and $\Pi$ is a  strict path connecting two 
different points $Q$ and $Q'$ of $\vn m$, then
$$\Sigma_{\Pi}(\alw)\ge c_{3,\alpha} \big(\bar \alpha_{min}\big)^m\eqno (5.5)$$
and consequently, $\wdal$ is a distance on $\vn{\infty}$.
\end{lem}

\begin{proof}
In view of Lemma
3.6, there exist a  strict path $\Pi'$ lying in $\vn m$ connecting
$Q$ and $Q'$ and $\bar\gamma_1,...,\bar\gamma_m\in\bar\Gamma$ such that

$$\Pi=D(\bar\gamma_1)\circ\cdots\circ D(\bar\gamma_k)(\Pi').$$

Note that, by Lemma 2.3, since $\Pi'$ lies in $\vn m$,
then  $|w(Q_{h-1,\Pi'},Q_{h,\Pi'})|$

$\le m$, thus 

$$\alw
(Q_{h-1,\Pi'},Q_{h,\Pi'})\ge  \big( \bar\alpha_{min}\big)^m 
\quad \forall\, h=1,...,n(\Pi'). \eqno(5.6)$$
Therefore, by Lemma 4.2 we have

$$ \Sigma_{\Pi}(\alw)=
\Sigma_{D(\bar\gamma_1)\circ\cdots\circ D(\bar\gamma_m)(\Pi')}(\alw)=$$
$$ \bar H\Big( \bar T\pal_{\bar\gamma_1}\circ\cdots\circ \bar T\pal_{\bar\gamma_m}
 \big(\witi\Sigma_{\Pi'}(\alw)\big)\Big)=$$
 $$=\sum\limits_{h=1}^{n(\Pi')}
\alw(Q_{h-1,\Pi'},Q_{h,\Pi'}) \, \bar H\Big( \bar T\pal_{\bar\gamma_1}
\circ\cdots\circ \bar T\pal_{\bar\gamma_m}
 \big(e_{\witi\iota(h,\Pi'),\witi w(h,\Pi')} )\big)\Big)$$
 $$\ge  \alw(Q_{0,\Pi'},Q_{1,\Pi'}) \bar H\Big( \bar T\pal_{\bar\gamma_1}
\circ\cdots\circ \bar T\pal_{\bar\gamma_m}
 \big(e_{\witi\iota(1,\Pi'),\witi w(1,\Pi')} )\big)\Big)$$
 $$\ge  c_{3,\alpha}\big( \bar\alpha_{min}\big)^m$$
 
 where we have used (5.6) and the definitions of $(\bar T,\bar\Gamma)$-u.p.
 and of  $\witi\Sigma_{\Pi'}(\alw)$. Therefore, (5.5) holds.
 In order to prove that $\wdal$ is  a distance,
 the only nontrivial fact to prove is that if $Q,Q'\in\vn{\infty}$, $Q\ne Q'$, then 
$\wdal(Q,Q')>0.$ But, since for some natural $m$ we have $Q,Q'\in\vn m$,
by (5.5) we have $\wdal(Q,Q')\ge c_{3,\alpha} \big( \bar\alpha_{min}\big)^m>0$.
 \end{proof}

\begin{lem}\label{5.5} If $\alpha$ is $(T,\Gamma)$ u.p.
there exists $c_{4,\alpha}>0$ such that for every $\tw\in W^*$
$$c_{4,\alpha} \alpha_{\tw}\diam_{\wdal}(\vn{\infty})\le\diam_{\wdal}(\vn{\infty}_{\tw})
\le \alpha_{\tw}\diam_{\wdal}(\vn{\infty}).\eqno(5.7)$$
\end{lem}
\begin{proof}
Let $Q,Q'\in \vn{\infty}_{\tw}$, $Q=\psi_{\tw}(P)$, $Q'=\psi_{\tw}(P')$ with $P,P'\in
\vn{\infty}$. We have

$$\wdal(Q,Q')=\min\{A,B\},$$
$$A:=\inf_{\Pi\in\cp}\Sigma_{\Pi}(\alw),\quad
\cp:=\psi_{\tw}\Big( \tpath_{P,P'}\Big),$$
$$B:=\inf_{\Pi\in\cp'}\Sigma_{\Pi}(\alw),\quad\cp'=\tpath_{Q,Q'}
\setminus \cp.$$

Let ${\tw}=(i_1...i_m)$. Since, in view of (5.3), 
$$A=\alpha_{\tw}\wdal(P,P')\eqno (5.8)$$

 the second inequality in (5.7) follows at once.
By Lemma 3.5, if $\Pi\in \cp'$, then there exist $l<m$ and a subpath $\Pi'$ of $\Pi$
connecting different points of $V_{i_1...i_{l+1}}$ entirely contained in 
$\vn{\infty}_{i_1...i_l}$. Thus, $\Pi'=\psi_{i_1...i_l}(\Pi'')$, where $\Pi''$
is  a path in $\vn{\infty}$ connecting two different points in $\von$.
By Lemma 5.4,
$\Sigma_{\Pi''}(\alw)\ge  c_{3,\alpha}\bar\alpha_{min}.$ Hence, 

$$\Sigma_{\Pi}(\alw)\ge \Sigma_{\Pi'}(\alw)=\alpha_{i_1...i_l}\Sigma_{\Pi''}(\alw)$$
$$\ge \alpha_{\tw} c_{3,\alpha}\bar\alpha_{min}
\ge c_{4,\alpha} \alpha_{\tw}\diam_{\wdal}(\vn{\infty})$$

if $c_{4,\alpha}\le \disp{c_{3,\alpha} \bar\alpha_{min}\over \diam_{\wdal}(\vn{\infty}) }$.
If moreover, $c_{4,\alpha}<1$, in view of (5.8), the first inequality in (5.7) holds.
\end{proof}

\begin{lem}\label{5.6}
For every polyratio $\balpha$ 
$$\Dal(Q,Q')\le \wdal(Q,Q')\le C_{1,\alpha}\Dal(Q,Q')\quad
\forall\, Q,Q'\in \vn{\infty}.\eqno (5.9)$$
\end{lem}
\begin{proof}
We can and do assume $Q\ne Q'$.
Let $\Pi\in \path_{Q,Q'}$. Then for every $h=1,...,n(\Pi)-1$ we have
$Q_{h,\Pi}\in  K_{\witi w(h,\Pi)}\cap K_{\witi w(h+1,\Pi)}$. Thus,
$\bar \cC:=\big(\witi w(1,\Pi),...,\witi w(n(\Pi),\Pi)\big)\in G'(Q,Q')$.
Moreover,

$$A(\bar \cC)=\sum\limits_{h=1}^{n(\Pi)}
 \balpha_{\witi w(h,\Pi)}=\Sigma_{\Pi}(\alw)$$

so that the first inequality in (5.9) follows. Let next
$\bar \cC\in G(Q,Q')$, and $\bar \cC=(\tw_1,...,\tw_n)$.
Suppose for the moment $n>1$.
For $h=0,...,n-2$
let $Q_h$ be an element of $K_{\tw_{h+1}}\cap K_{\tw_{h+2}}$, thus,
$\tw_{h+1}$ and $\tw_{h+2}$ being incomparable, by Lemma 2.2
we have $Q_h\in V_{\tw_{h+1}}\cap V_{\tw_{h+2}}$. 
it follows that $\Pi$, defined by

$$\Pi:=(Q_0,...,Q_{n-2})$$

is a weak path, not necessarily a path. In fact, for every $h=1,...,n-2$
$Q_{h-1},Q_h\in  V_{\tw_{h+1}}$, so that either
$Q_{h-1}\ne Q_h$ and 
$\witi w(h,\Pi)=\tw_{h+1}$, or $Q_{h-1}= Q_h$, and, thanks to the
convention $\balpha_{\witi w(h,\Pi)}=0$ (see Remark 4.1), we have

$$\wdal(Q_0,Q_{n-2})\le \Sigma_{\Pi}(\alw)=\sum\limits_{h=1}^{n-2}
 \balpha_{\witi w(h,\Pi)}\le  \sum\limits_{h=2}^{n-1} \alpha_{\tw_h}.$$

 Moreover, since 
$Q\in K_{\tw_1}\cap \vn {\infty}= \vn{\infty}_{\tw_1}$, we have
$Q,Q_0\in \vn{\infty}_{\tw_1}$, hence by Lemma 5.3 
$\wdal(Q,Q_0)\le C_{1,\balpha} \balpha_{\tw_1}$. Similarly,
$\wdal(Q_{n-2},Q')\le C_{1,\balpha} \balpha_{\tw_n}$. In conclusion

$$\wdal(Q,Q')\le \wdal(Q,Q_0)+\wdal(Q_0,Q_{n-2})+\wdal(Q_{n-2},Q')$$
$$\le C_{1,\balpha} \balpha_{\tw_1}+\,
\sum\limits_{h=2}^{n-1} \alpha_{\tw_h}\ +C_{1,\balpha} \balpha_{\tw_n}$$
$$\le C_{1,\balpha}\Big(\balpha_{\tw_1}+\,
\sum\limits_{h=2}^{n-1} \alpha_{\tw_h}+\ \balpha_{\tw_n}\Big)=
C_{1,\balpha}    \sum\limits_{h=1}^{n} \alpha_{\tw_h}=
C_{1,\balpha} A(\bar \cC),.$$

We have proved the inequality 
$$\wdal(Q,Q')\le C_{1,\balpha} A(\bar \cC)\eqno (5.10)$$
when $n>1$. However,  (5.10) also holds in the case $n=1$. In fact, in this case,
$Q,Q'\in \vn{\infty}\cap K_{\tw_1}=\vn{\infty}_{\tw_1}$.  Thus,
$$\wdal(Q,Q')\le \diam_{\wdal}\big(\vn{\infty}_{\tw_1}\big)\le 
 C_{1,\balpha} \alpha_{\tw_1}=       C_{1,\balpha} A(\bar \cC).$$
Since (5.10) holds foe every
$\bar \cC\in G(Q,Q')$,  the second inequality in (5.9) easily follows.
\end{proof}

\begin{theo}\label{5.7}
If the polyratio $\alpha$ is $(T,\Gamma)$-u.p., then $\Dal$
is both $\balpha$-self-similar and $\balpha$-scaling. Thus,
$\balpha$ is both metric and as. metric. Moreover,
$\Dal$ induces on $K$ the same topology as the original distance.
\end{theo}
\begin{proof}
By Lemma 5.6 and Lemma 5.4, if $Q,Q'\in\vn{\infty}$, $Q\ne Q'$, then
$\Dal(Q,Q')\ge {1\over C_{1,\balpha}} \wdal(Q,Q')>0$. Thus,
by Lemma 5.2, $\Dal$ is an \asd \ on $K$ and  
induces on $K$ the same topology as the original distance.
Therefore, in view of Lemma 2.6 ii)
$\vn{\infty}_{\tw}$ is dense in $K_{\tw}$ for every
$\tw\in W^*$ also with respect to $\Dal$. Now,
$\balpha$ is as. metric by Lemmas 5.5 and 5.6.
\end{proof}


\begin{rem}
\text{\rm  Note that it follows from Lemma 4.3 i) that,$\,$if  we have}

$H\big(T\pal_{\gamma }(e_{\iota}) \big)\ge 1\ \ \,  \forall\,\iota\in\widehat J\ \
\forall\, \gamma\in\Gamma,$
\text{\rm then $\alpha\,$is$\,(T,\Gamma)$-u.p.. More
 precisely,}
 
 \text{\rm we can prove
   by induction that } 
$$ H\big(T\pal_{\gamma_{1} }\circ\cdots\circ T\pal_{\gamma_{m}}(e_{\iota})\big)\ge 1
\quad \forall\, \iota\in\widehat J
\ \ \forall\, \gamma_1,...,\gamma_m\in\Gamma.\eqno(5.11)$$
\text{\rm In fact, (5.11) is trivial for $m=0$. Also, if it holds for $m$, then}
$$ H\Big(T\pal_{\gamma_{1} }\circ\cdots\circ T\pal_{\gamma_{m+1}}(e_{\iota})\Big)
=\sum\limits_{\iota'\in \widehat J}
T\pal_{\gamma_{m+1} } (e_{\iota})_{\iota'} H
\Big( T\pal_{\gamma_{1} }\circ\cdots\circ T\pal_{\gamma_{m}}(e_{\iota'})\Big)$$.
$$ \ge\sum\limits_{\iota'\in \widehat J}
T\pal_{\gamma_{m+1} } (e_{\iota})_{\iota'}
\eqno \text{\rm by\ (5.11)}$$
$$=H\Big(T\pal_{\gamma_{m+1} }(e_{\iota})\Big)\ge 1,$$
\text{\rm where $\,$in  $\,$the $\,$last $\,$equality $\,$we 
$\,$have $\,$used $\,$Lemma 4.3 i) $\,$with $m=0$,} 

\text{\rm since$\, H(e_{\iota'})=1,\,$and (5.11) holds for$\,m+1.$
As$\,$a$\,$consequence$\,$of$\,$(5.11),}

\text{\rm and of Corollary 4.4,  if $\,\alpha_i\ge {1\over 2}\,$ for every $\, i=1,...,k$, then
there exists} 

\text{\rm an $\alpha$-scaling distance on $K$.}

\text{\rm \qua  In fact, if $\,\gamma(\iota)\,$ is not
a strong $\,\von$-path,
 then$\,\witi w(h,\gamma(\iota))=\empti\,$, thus}
 
 \text{\rm $\,\alpha_{\witi w(h,\gamma(\iota))}
=1$ for at $\,$least one $\,h=1,...,n(\gamma(\iota))$.
Therefore, in view of} 

\text{\rm (4.4),  (5.11) holds. If on the contrary, $\gamma(\iota)$ is 
a strong $\von$-path, as it}

\text{\rm  connects two different points 
of $\vo$,  the sum in$\,(4.4)\,$has at least two}

\text{\rm   summands. Indeed, $\gamma(\iota)\,$connects two$\,$different
points of $\vo$,$\,$and  two} 

\text{\rm consecutive vertices of$\, \gamma(\iota)$, by definition
of a$\,$strong$\, \von$-path, lie in a}

\text{\rm a same $1$-cell, which, by (2.8),  cannot contain more than
one point of}

\text{\rm    $\vo$.  Moreover, in (4.4) we have
$\witi w(h,\gamma(\iota))=i_h$ for some $i_h=1,...,k$,}

\text{\rm  and, $\,$by our $\,$assumptions,$\,$ such summands 
$\,$are not less than $\,{1\over 2}\,$,$\,$  thus}
 
 \text{\rm (5.11)  holds in any case.} 
\end{rem}

\section{Examples}

In view of the results of Section 5, the fact whether
a given polyratio $\alpha$ is metric (or as. metric, which is the same)
on  $K$ is reduced
to the problem whether $\alpha$ is   $(T,\Gamma)$-u.p..
In turns, such a problem is strictly related to the notion of {\it joint spectral radius}
or better {\it joint spectral subradius}. More precisely, it is related to the fact that
the joint spectral subradius is greater than or equal to $1$. However,
the notion of $(T,\Gamma)$-u.p. in general is not perfectly equivalent
to having the     joint spectral subradius  greater than or equal to $1$.
Note that the determination of the joint spectral radius, as well as 
of the joint spectral subradius is in general difficult.
In this Section, we will discuss some explicit necessary and sufficient conditions
for $\alpha$ being $(T,\Gamma)$-u.p.. However, such conditions
require the existence of some special paths, and this occurs only
on some fractals having a rather simple structure.

\qua  If ${\bf\Pi}=(\Pi(1),...,\Pi(l))$ is a finite sequence of paths we put

$$\Sigma_{\bf\Pi}(\alw)=\sum\limits_{\beta=1}^l \Sigma_{\Pi_{\beta}}(\alw).$$

 We say that the subpaths $\Pi(1),...,\Pi(l)$ of $\Pi$ are
{\it separated}, if they have length greater than $1$, and moreover,
the intervals

$$\big]n_1(\Pi,\Pi(\beta)),n_2(\Pi,\Pi(\beta))\big], \quad \beta=1,...,l,$$

 are mutually disjoint,
roughly speaking this means that they  have no common edge. 
In this setting we say that
 ${\bf \Pi}:=\big(\Pi(1),...,\Pi(l)\big)$ is  a {\it multisubpath} or shortly multisp. of $\Pi$.
%
We say that a sequence $ \Pi':=(Q_{h_0,\Pi},...,Q_{h_s,\Pi})$ is a {\it refinement} of $\Pi$ if
$0= h_0<h_1<\cdots<h_s=n(\Pi)$, and moreover
$Q_{h_{l+1},\Pi}=Q_{h_{l}+1,\Pi}$ for every for every $l=0,...,s-1$. Thus,
$\witi\Pi$ is a path.  Note that if
$\Pi$ is a {\it strict} path, then every refinement of $\Pi$ amounts to $\Pi$ itself.

 \qua Suppose now there exist finitely many $\von$
paths $\Pi_1,...,\Pi_r$ such that

\smallskip
(i) For every $s=1,...,r$ there exists  $\iota_s\in \widehat J$  such that
$\Pi_s$ is a  strict and strong
 $(\iota_s,\von)$-path, and also $\witi\iota(h,\Pi_s)=\iota_s$ for every $h$.

\smallskip
(ii) There exists $\witi{\,m\,}\in\bna$ such that 
every  path $\Pi$ connecting two
different  points of $\vo$ and such that
$|\witi w(h,\Pi)|\ge \witi{\,m\,}$for every $h=1,..., n(\Pi)$, also
 contains a subpath of the form $\psi_{\tw}(\bar\Pi)$
where $\bar\Pi$ is a $\iota_s$-path for some $s=1,...,r$, and $\tw\in W_
 {\witi{\,m\,}}$.

\smallskip
(iii) For every $s=1,...,r$ and  every strong $\iota_s$-path $\Pi$,  there 
exist $s'=1,..,r$ and  a multisp. 
$\widehat{\bf \Pi}:=(\widehat \Pi_1,...,\widehat\Pi_{n(\Pi_{s'})})$ of $\Pi$ with refinements
 $\widehat\Pi'_{\beta}$ of $\widehat\Pi_{\beta}$, 
 $s(\beta)=1,...,r$
such that
$\widehat\Pi'_{\beta}=\psi_{\witi w(\beta,\Pi_{s'})}(\widehat\Pi''_{\beta})$ 
where $\widehat\Pi''_{\beta}$ is a $\iota_{s(\beta)}$-path,
and we put ${\bf{\widehat \Pi'}}=(\widehat\Pi'_1,...,\widehat\Pi'_{n(\Pi_{s'})})$.
\smallskip

\begin{lem}\label{6.1}
Under assumption (iii), if $\Pi$ is a strong $\iota_s$-path, $s=1,...,r$, 
and $\Sigma_{\widehat\Pi''_{\beta}}(\alw)\ge 1$ for every $\beta=1,...,n(\Pi_{s'})$, then
$$\Sigma_{\Pi}(\alw)\ge \Sigma_{\Pi_{s'}}(\alw).$$
\end{lem}

\begin{proof}
Note that, if $\Pi'$ is a refinement of $\Pi$
 we have 

$$ \Sigma_{\Pi'}(\alw)\le \Sigma_{\Pi}(\alw).\eqno (6.1)$$
In fact, using the previous notation we have 
$$\Sigma_{\Pi'}(\alw)=
\sum\limits_{l=1}^s\alw(Q_{h_{l-1},\Pi},Q_{h_l,\Pi})$$
$$=\sum\limits_{l=1}^s\alw(Q_{h_{l-1},\Pi},Q_{h_{l-1}+1,\Pi})\le
\sum\limits_{h=1}^{n(\Pi)}\alw(Q_{h-1,\Pi},Q_{h,\Pi})
=\Sigma_{\Pi}(\alw).$$

Moreover,

$$\Sigma_{\bf\widehat \Pi'}(\alw)\le \Sigma_{\Pi}(\alw). \eqno (6.2)$$

In fact, on one hand, by (6.1) we have 

$$ \Sigma_{\bf\widehat\Pi'}(\alw)=\sum\limits_{\beta=1}^{n(\Pi_{s'})} 
\Sigma_{\widehat\Pi'_{\beta}}(\alw)
\le \sum\limits_{\beta=1}^{n(\Pi_{s'})}\Sigma_{\widehat\Pi_{\beta}}(\alw)=
 \Sigma_{\bf\widehat\Pi}(\alw).$$

On the other, by (4.2$'$) we have
$$\Sigma_{\bf\widehat\Pi}(\alw)=\sum\limits_{\beta=1}^{n(\Pi_{s'})}
 \Sigma_{\widehat\Pi_{\beta}}(\alw)$$
$$=  \sum\limits_{\beta=1}^{n(\Pi_{s'})}\sum\limits_{h=1}^{n(\widehat\Pi_{\beta})}
 \alpha_{\witi w(h,\widehat\Pi_{\beta})}                     $$
$$=\sum\limits_{\beta=1}^{n(\Pi_{s'})}\sum\limits_{h=1}^{n(\widehat\Pi_{\beta})}
\alpha_{\witi w\big(h+n_1(\Pi,\widehat\Pi_{\beta}),\Pi\big)}
\le \sum\limits_{h=1}^{n(\Pi)}
\alpha_{\witi w(h,\Pi)}=\Sigma_{\Pi}(\alw).$$  

In fact we can easily verify that   $n(\widehat \Pi_{\beta})=
n_2(\Pi,\widehat\Pi_{\beta})-
n_1(\Pi, \widehat\Pi_{\beta})$,  and for every $h=1,...,n(\widehat\Pi_{\beta})$ we have
$\witi w(h,\widehat\Pi_{\beta})= \witi w\big(h+
n_1(\Pi,\widehat\Pi_{\beta}),\Pi\big)$,
$h+n_1(\Pi,\widehat\Pi_{\beta})\in 
 \big]n_1(\Pi,\widehat\Pi_{\beta}),n_2(\Pi,\widehat\Pi_{\beta}\big]$,
and moreover, the intervals  
$\big]n_1(\Pi,\widehat\Pi_{\beta}),n_2(\Pi,\widehat\Pi_{\beta})\big]$ are mutually 
disjoint. Thus, (6.2) is proved. 

Finally,  it is easy to verify that
for every path $\Pi$ and every $\tw\in W^*$ we have
$$\Sigma_{\psi_{\tw}(\Pi)}(\alw)=\alpha_{\tw} \Sigma_{\Pi}(\alw). \eqno (6.3)$$
By (6.2) and (6.3), we have

$$\Sigma_{\Pi}(\alw)\ge \Sigma_{\bf\widehat \Pi'}(\alw)$$
$$=\sum\limits_{\beta=1}^{n(\Pi_{s'})}  \Sigma_{\widehat \Pi'_{\beta}}(\alw)$$
$$=\sum\limits_{\beta=1}^{n(\Pi_{s'})} \alpha_{\witi w(\beta,\Pi_{s'})}
\Sigma_{\widehat \Pi''_{\beta}}(\alw)$$
$$\ge \sum\limits_{\beta=1}^{n(\Pi_{s'})} \alpha_{\witi w(\beta,\Pi_{s'})}=
\Sigma_{\Pi_{s'}}(\alw).$$
\end{proof}

\begin{lem}\label{6.2}
Under the previous assumptions,
suppose moreover that   for every $s'=1,...,r$ we have
$\Sigma_{\Pi_{s'}}(\alw)\ge 1$.
Let $\Pi$ be a  path connecting two different points of $\vo$.
Then 
$$\Sigma_{\Pi}(\alw)\ge \bar\alpha_{min}^{\witi{\,m\,}}. \eqno(6.4)$$

\end{lem}
\begin{proof}

Suppose for the moment $\Pi$ is a $\iota_s$-path for some $s=1,...,r$, and prove

$$\Sigma_{\Pi}(\alw)\ge 1. \eqno(6.5)$$

Let $\bar m=\bar m(\Pi)$ be 
the maximum $m$ such that
there exists a point in $\Pi$ that lies in $\vn m\setminus \vn {m-1}$ (here $\vn{-1}:=\empti$).
Note that we have $Q_{h,\Pi}\in \vn {\bar m}$ for every $h$ by the definition of
$\bar m$. 

\qua We prove (6.5) by induction on $\bar m$. If $\bar m=0$, then 
$\Pi$ is a $\vo$-path, and thus 
 $\alw(Q_{h-1,\Pi}, Q_{h,\Pi})=1$ for every $h=1,...,n(\Pi)$, 
and (6.5) is trivial.

\qua Suppose now (6.5) holds for $\bar m\le m_1\in\bna$ and prove it holds also for $m_1+1$.
Suppose $\bar m=m_1+1$. If $\Pi$ is not strong, then
then there exists $\bar h=1,...,n(\Pi)$ such that
$Q_{\bar h-1,\Pi}, Q_{\bar h,\Pi}\in\vo$, thus $\alw\big(Q_{\bar h-1,\Pi}, Q_{\bar h,\Pi}\big)=1$
and (6.5) is trivial. So, suppose $\Pi$ is  strong.
For every $\beta=1,...,n(\Pi_{s'})$ and $h=1,...,n(\widehat\Pi''_{\beta})$
let $m_{\beta,h}\in\bna$ be such that 

$$Q_{h,\widehat \Pi''_{\beta}}\in \vn{m_{\beta,h}}\setminus \vn{m_{\beta,h}-1}.$$

Thus if $m_{\beta,h}>0$, we have

$$Q_{h,\widehat\Pi'_{\beta}}=\psi_{\witi w(\beta,\Pi_{s'})}\big(
Q_{h,\widehat\Pi''_{\beta}}\big)\in \vn{m_{\beta,h}+1}\setminus \vn{m_{\beta,h}}.$$

In fact, since $\Pi_{s'}$ is a strong $\von$-path, then $\witi w(\beta,\Pi_{s'})$ has length
$1$, and we can use Corollary 2.5. Now, since $Q_{h,\widehat\Pi'_{\beta}}$ is also a vertex
of $\Pi$, by definition we have $m_{\beta,h}+1\le \bar m=m_1+1$, thus
$m_{\beta,h}\le m_1$, and this is trivially valid also if $m_{\beta,h}=0$.
By definition, $\bar m(\widehat \Pi''_{\beta})\le m_1$, and by our inductive hypothesis,
$\widehat \Pi''_{\beta}$ satisfies (6.5).   By Lemma 6.1, $\Sigma_{\Pi}\alw\ge 
\Sigma_{\Pi_{s'}}(\alw)\ge 1$, thus $\Pi$ satisfies (6.5).

\qua Now, we prove (6.4). If there exists $\bar h=1,...,n(\Pi)$ such that
$|\witi w(\bar h,\Pi)|<\witi{\, m\,}$, then
$\Sigma_{\Pi}(\alw)\ge \alpha_{\witi w(\bar h,\Pi)}\ge 
\bar\alpha_{min}^{\witi{\,m\,}}$, and (6.4) holds.
In the opposite case, by assumption (ii), using notation thereof,
we have $\alpha_{\tw}\ge  \bar\alpha_{min}^{\witi{\,m\,}}$
and  $\Sigma_{\bar \Pi}(\alw)\ge 1$ by (6.5) since $\bar\Pi$
 is a $\iota_s$-path. Also using (6.3), we obtain

$$\Sigma_{\Pi}(\alw)\ge \Sigma_{\psi_{\tw}(\bar \Pi)}(\alw)=
\alpha_{\tw}\Sigma_{\bar \Pi}(\alw)\ge \bar\alpha_{min}^{\witi{\,m\,}}$$

\end{proof}

\begin{theo}\label{6.3}
Under the previous assumptions,  $\alpha$ is  a metric polyratio 
on $K$ if and only if  $\Sigma_{\Pi_s}(\alw)\ge 1$ for every $s=1,...,r$.
\end{theo}
\begin{proof}
We use Corollary 5.8. Suppose there exists an $\alpha$-scaling distance
on $K$, hence  $\alpha$
is $(T,\Gamma)$-u.p.. Let $\gamma\in \Gamma$ such that
 $\gamma(\iota_s)=\Pi_s$. 
 Putting $\Pi=D(\gamma)(P_{\iota_s})$ and $\iota_s=(j_1,j_2)$, we obtain

$$\Pi=D(\gamma)(P_{\iota_s})=D(\gamma)(P_{j_1},P_{j_2})=\gamma(j_1,j_2)=\Pi_s.$$
 
Hence, in view of (4.1) and Lemma 4.2 we get

$$T\pal_{\gamma}(e_{\iota_s})=T\pal_{\gamma}\big( \widehat\Sigma_{\iota_s}(\alw)\big)$$

 $$=\widehat\Sigma_{D(\gamma)(P_{\iota_s})}(\alw)$$
$$=\sum\limits_{h=1}^{n(\Pi)}\alw (Q_{h-1,\Pi},Q_{h,\Pi}) e_{\witi \iota(h,\Pi)}=$$
$$ \sum\limits_{h=1}^{n(\Pi)}\alw (Q_{h-1,\Pi},Q_{h,\Pi}) e_{\witi \iota(h,\Pi_s)}=$$
$$\Big( \sum\limits_{h=1}^{n(\Pi_s)} \alw (Q_{h-1,\Pi_s},Q_{h,\Pi_s})\Big) e_{\iota_s}=
 \Big(\Sigma_{\Pi_s}(\alw) \Big)e_{\iota_s}.
 $$
 
 Consequently,
 
 $$ \big(T\pal_{\gamma})^n(e_{\iota_s})=
\Big( \Sigma_{\Pi_s}(\alw)\Big)^n  e_{\iota_s}
 $$
 
 for every positive integer $n$, thus, since $\alpha$  is $(T,\Gamma)$-u.p.,
 we must have $\Sigma_{\Pi_s}(\alw)\ge 1$.
For the converse, suppose $\Sigma_{\Pi_s}(\alw)\ge 1$ 
for every $s=1,...,r$, and prove
that $\alpha$ is $(T,\Gamma)$-u.p.
Let $\gamma_1,...,\gamma_n\in\Gamma$ and let $\iota\in\widehat J$.  
Of course, the path 
$D(\gamma_1)\circ\cdots\circ D(\gamma_n)(P_{\iota})$ connects two points of $\vo$.
 Thus, 
by Lemma 6.2,  we have
$$\Sigma_{D(\gamma_1)\circ\cdots\circ D(\gamma_n)(P_{\iota})}(\alw)\ge 
\bar\alpha_{min}^{\witi{\,m\,}}.$$

By Lemma 4.2 ii) we have

$$\ H\Big(T\pal_{\gamma_1}\circ\cdots\circ  
T\pal_{\gamma_n}(e_{\iota})\Big)$$
$$= H\Big(T\pal_{\gamma_1}\circ\cdots\circ  T\pal_{\gamma_n}
 \big(\widehat\Sigma_{P_{\iota}}(\alw)\big)\Big)=$$
$$H\Big(\widehat\Sigma_{D(\gamma_1)\circ\cdots\circ D(\gamma_n)(P_{\iota})}(\alw)\Big)$$
$$=\Sigma_{D(\gamma_1)\circ\cdots\circ D(\gamma_n)(P_{\iota})}(\alw)\ge
\bar\alpha_{min}^{\witi{\,m\,}}
$$
 
 and thus $\alpha$ is $(T,\Gamma)$-u.p..
\end{proof}

\smallskip
We now apply Theorem 6.3 to two examples. For the Gasket
with three $1$-cells $V_1$, $V_2$, $V_3$, let $P_1$, $P_2$, $P_3$ be the three
fixed points of the maps, let $Q_{j_1,j_2}=Q_{j_2,j_1}:=\psi_{j_1}(P_{j_2})
=\psi_{j_2}(P_{j_1})$ when $j_1,j_2=1,2,3$, $j_1\ne j_2$.
It is simple to verify
that the six paths of the form
$(P_{j_1}, Q_{j_1,j_2}, P_{j_2})$, $j_1,j_2=1,2,3$, $j_1\ne j_2$,
satisfy (i), (ii), (iii), with $\witi{\, m\,}=0$,
thus, in view of Theorem 6.3,
 $\alpha$ is a metric polyratio
 on the Gasket  if and only if we have

$$\alpha_1+\alpha_2\ge 1, \quad \alpha_1+\alpha_3\ge 1
\quad \alpha_2+\alpha_3\ge 1. \eqno(6.6)$$

\qua  In the Vicsek set,
let $\psi_i$, $i=1,2,3,4,5$ be the contractions defining it, and we order
them in such a way that
$\psi_5$ the contraction that fixes the center. Let $P_j$ be the fixed points
of $V_j$ for $j=1,2,3,4$. Choose the order so that $P_1$ is opposite
to $P_3$ and $P_2$ is opposite to $P_4$. Let $Q_j$ the only point
in $V_j\cap V_5$. Now we take the four
paths of the form $\Pi_j:=(P_j, Q_j,Q_{j'},P_{j'})$ when $j=1,2,3,4$
and $P_{j'}$ is opposite to $P_j$.  It is simple to verify that such paths satisfy
(i), (ii), (iii), with $\witi{\, m\,}=1$. Thus, by Theorem 6.3,
$\alpha$ is a metric polyratio
 on the Vicsek set   if and only if  we have

 $$\alpha_1+\alpha_3+\alpha_5\ge 1, 
 \quad \alpha_2+\alpha_4+\alpha_5\ge 1.\eqno (6.7)$$

Considerations similar could be extended to other fractals.
However, in order to have simple conditions like (6.6) or (6.7)
the structure of the fractal should be simple. In most cases,
I expect that it could be hard to give simple necessary and sufficient conditions
for having an $\alpha$-scaling distance.

\end{document}